\documentclass[11pt]{article}

\usepackage{amsmath,amssymb,fullpage}

\usepackage{enumerate}

\usepackage{graphicx}

\usepackage{makeidx}
\usepackage[utf8]{inputenc}
\usepackage{wrapfig}
\usepackage{amsmath}
\usepackage{amsfonts}
\usepackage{mathrsfs}
\usepackage{amssymb}
\usepackage{latexsym}
\usepackage{amsbsy}
\usepackage{color}
\usepackage{bbm}

\usepackage[colorlinks,citecolor=blue,urlcolor=blue]{hyperref}

\usepackage{mathrsfs}

\usepackage{wasysym}

\providecommand{\otherindexspace}[1]{}

\usepackage{amsthm}

\newtheorem{theorem}{Theorem}[section]

\newtheorem{lemma}[theorem]{Lemma}
\newtheorem{proposition}[theorem]{Proposition}
\newtheorem{remark}[theorem]{Remark}

\newtheorem{definition}[theorem]{Definition}

\newtheorem{assumption}[theorem]{Assumption}

\numberwithin{equation}{section}

\def\p{\partial}

\def\vp{\varepsilon}

\def\cal#1{\mathcal{#1}}

\def \H{\mathbb {H}}
\def \R{\mathbb {R}}

\def\dd{\displaystyle}
\usepackage{fancyhdr}
\newcommand{\dproof}{\noindent {Proof.} \quad}
\newcommand{\fproof}{\hfill $\square$ \bigskip}

\usepackage{graphics}
\usepackage{graphicx}

\DeclareGraphicsExtensions{.eps,.bmp,.jpg,.pdf,.mps,.png,.gif}

\makeatletter
\def\titre{\@title}
\makeatother

\title{A Weak Dynamic Programming Principle for Combined Optimal Stopping / Stochastic Control with 
${\cal E}^f$-expectations}

\author{Roxana Dumitrescu\thanks{Institut für Mathematik, Humboldt-Universität zu Berlin, Unter den Linden 6, 10099 Berlin, Germany,
 email: \textbf{roxana@ceremade.dauphine.fr}. The research leading to these results has received funding from the Region Ile-de-France.} \and Marie-Claire Quenez \thanks{LPMA,
Université Paris 7 Denis Diderot, Boite courrier 7012, 75251 Paris cedex 05, France, email: \textbf{quenez@math.univ-paris-diderot.fr}} \and  Agnès Sulem
\thanks{ INRIA Paris, 3 rue Simone Iff, CS 42112, 75589 Paris Cedex 12, France, and Université Paris-Est, email: \textbf{agnes.sulem@inria.fr}}}

\begin{document}



\date{\today}

\maketitle

\begin{abstract}
 We study a combined optimal control/stopping problem under a nonlinear expectation ${\cal E}^f$ induced by a  BSDE with jumps, in a Markovian framework. The terminal reward function is only supposed to be Borelian. The value 
function $u$ associated with  this problem is generally irregular. 
We first establish a {\em sub- (resp. super-) optimality principle of dynamic programming}  involving its {\em upper- (resp. lower-) 
semicontinuous envelope} $u^*$ (resp. $u_*$). This result, called {\em weak} dynamic programming principle (DPP), extends 
that obtained  in \cite{BT} in the case of a classical expectation 
to the  case of an ${\cal E}^f$-expectation
and  Borelian terminal reward function. 
Using this  {\em weak} DPP, 
 we then prove 
  that  $u^*$ (resp. $u_*$) is a {\em viscosity sub- (resp. super-) solution} of  a nonlinear 
  Hamilton-Jacobi-Bellman variational inequality.

 \end{abstract}


\bigskip

\textbf{Key-words:} Markovian stochastic control, mixed optimal control/stopping, non linear expectation, backward stochastic differential equation, weak dynamic programming principle, 
Hamilton-Jacobi-Bellman variational inequality, viscosity solution, ${\cal E}^f$-expectation.


\section{Introduction}


  Markovian stochastic control problems on a given horizon of time $T$ can typically  be written as 
\begin{equation}\label{linear}
u(0,x)= \sup_{\alpha \in \mathcal{A}}\mathbb{E}[\int_0^T f(\alpha_s,X_s^\alpha)ds+g(X_T^\alpha)], 
\end{equation}
where $\mathcal{A}$ is a set of admissible control processes $\alpha_s$, and 
$(X_s^\alpha)$ is a controlled process of the form 
$$X_s^\alpha=x+\int_0^sb(X_u^\alpha,\alpha_u)du+\int_0^s\sigma(X_u^\alpha,\alpha_u)dW_u+\int_0^s \int_{\mathbb{R}^n}\beta(X_u^\alpha,\alpha_u,e)\Tilde{N}(du,de).$$ 
The random variable $g(X_T^\alpha)$  may represent a terminal reward and $f(\alpha_s,X_s^\alpha)$ an instantaneous reward process. 
Formally, for all initial time $t$ in $[0,T]$ and initial state $y$, the associated value function is defined by
\begin{equation}\label{linearvt}
u(t,y) = \sup_{\alpha \in \mathcal{A}}\mathbb{E}[\int_t^T f(\alpha_s,X_s^\alpha)ds+g(X_T^\alpha)\,\mid \, 
X_t^\alpha = y
]. 
\end{equation}
The dynamic programming principle  can formally   be stated as
\begin{equation} \label{linearddp}
u(0,x)= \sup_{\alpha \in \mathcal{A}}\mathbb{E}[\int_0^t f(\alpha_s,X_s^\alpha)ds+u(t, X_t^\alpha)], 
\quad {\rm for }  \; t \; {\rm in  } \; [0, T]. 
\end{equation}
This principle is classically  established 
under  assumptions which ensure that the value function $u$ satisfies some regularity properties.  From this principle, it can be derived that the value function is a  viscosity solution of 
the associated Hamilton-Jacobi-Bellman (HJB) equation. Similar results are obtained for optimal stopping 
and mixed optimal stopping/control problems.

The case of a discontinuous value function and its links with viscosity solutions has been studied  for 
 deterministic control in the eighties.   Barles and Perthame (1987)  study in \cite{BP} a deterministic optimal stopping problem with a reward map $g$ only supposed to be Borelian. To this purpose, they  introduce a notion of viscosity 
 solution which extends the classical one to the discontinuous case: a function $v$ is said to be a {\em weak} viscosity solution 
 of the  HJB equation if its upper semicontinuous (u.s.c.) envelope, denoted by $v^*$,  is a {\em viscosity sub-solution} of this PDE, and if its the lower semicontinuous (l.s.c.) envelope, denoted by $v_*$, is a {\em viscosity super-solution} of this equation.
 Then, by the classical dynamic programming principle provided in the previous literature, they get  that the u.s.c. envelope $u^*$ of the value function satisfies a  {\em sub-optimality principle} 
 in the sense of Lions and Souganidis (1985) in \cite{LS}. Using this sub-optimality principle, they then 
show that
  $u^*$ is a  {\em viscosity sub-solution} of the  HJB equation. Moreover, using the fact that the l.s.c  envelop $u_*$ of $u$ 
 is the value function of a relaxed problem, they show that 
  $u_*$ is a  {\em  viscosity super-solution}, and thus get that $u$ is a  {\em weak} viscosity solution of the HJB equation. 
  They stress that in general, the {\em weak} viscosity solution of this PDE  is not unique. However, under a regularity assumption on the reward $g$, 
   by using the control formulae, they obtain that the u.s.c. envelope $u^*$ of the value function is the unique u.s.c. viscosity solution of the HJB equation
(see Remark \ref{caracfaible} for additional references and comments).

More recently, in a stochastic framework, Bouchard and Touzi (2011) (see  \cite{BT}) have proven a  {\em weak} dynamic programming  principle (DPP)  when 
the terminal map $g$ is irregular: they prove  that  the value function $u$ satisfies  a sub-optimality principle of dynamic programming involving its u.s.c. envelope $u^*$, 
and under an additional regularity (lower semi continuity) assumption of the reward $g$, they obtain 
a super-optimality principle involving the l.s.c. envelope $u_*$.
Then, using the sub-optimality principle, they derive that $u^*$ is a viscosity subsolution of the associated HJB equation. Moreover, when $g$ is l.s.c.\,, using  the super-optimality principle, they 
show that $u_*$ is a viscosity super-solution, and thus get that $u$ is a {\em weak} viscosity solution of this PDE in the same sense as above (or  \cite{BP}).
  A {\em weak} dynamic programming  principle has been further  established, when $g$ is {\em l.s.c.} for problems 
  with state constraints by Bouchard and Nutz (2012) in  \cite{BN},  and, when $g$ is {\em continuous}, for zero-sum stochastic  games  by Bayraktar and Yao (2013) in \cite{BY}. 

In this paper we  are interested  in  generalizing these results to the case when $g$ is only Borelian and 
when the linear expectation $\mathbb{E}$ is replaced by a nonlinear expectation induced by a Backward Stochastic Differential Equation (BSDE) with jumps. 
Typically, such  problems in the Markovian case can be formulated as
 \begin{equation} \label{eqintroa}
 \sup_{\alpha \in \mathcal{A}} \mathcal{E}_{0, T}^\alpha[g(X_T^{\alpha})],
\end{equation}
where  $\mathcal{E}^\alpha$ is the nonlinear expectation 
  associated with  a BSDE with jumps with  controlled driver $f(\alpha_t,X_t^\alpha,y,z,k)$. 
Note that  Problem \eqref{linear} is  a particular case of \eqref{eqintroa} when the driver $f$ does not depend on the solution of the BSDE, that is when $f(\alpha_t,X_t^\alpha,y,z,k) \equiv f(\alpha_t,X_t^\alpha)$.

We first provide a {\em weak} dynamic programming principle involving the u.s.c. and l.s.c. envelopes of the value function. To this purpose, we prove some preliminary results, in particular some measurability and
``splitting" properties. 
No regularity condition on $g$ is required to obtain the sub and super-optimality principles, which is not the case in the previous literature in the stochastic case, even with a classical expectation 
(see \cite{BT}, \cite{BN} and \cite{BY}). Using this {\em weak} DPP, 
we then show that the value function, 
 which is generally neither u.s.c. nor l.s.c.\,, is a  {\em weak} viscosity solution (in the sense of \cite{BP})
of an associated nonlinear  HJB equation.

Moreover,  in this paper, we consider the  combined problem  when there is an
 additional control in the form of a stopping time. We thus consider 
mixed generalized optimal  control/stopping problems of the form
\begin{align}\label{eqintro}
\sup_{\alpha \in \mathcal{A}} \sup_{\tau \in \mathcal{T}}  \mathcal{E}_{0, \tau}^\alpha[\bar{h}(\tau, X_\tau^{\alpha})],
\end{align}
where $\mathcal{T}$ denotes the set of stopping times with values in $[0,T]$, and $\bar{h}$ is an irregular reward function. 

Note that in the literature on BSDEs, some papers (see e.g. Peng (1992) \cite{Peng92}, Li -Peng (2009) \cite{LP}, Buckdahn and Li (2008) \cite{Buckdahn} and Buckdahn and Nie (2014) \cite{Buckdahn1}) 
study stochastic control problems with nonlinear $ \mathcal{E}$-expectation in the 
{\em continuous} case (without optimal  stopping). Their approach is different from ours and relies on  the continuity assumption of the reward function.
 The paper is organized as follows: in Section \ref{sec2}, we formulate our generalized mixed  control-optimal stopping problem. 
Using results on reflected BSDEs (RBSDEs), we express this problem as an optimal  control problem for RBSDEs. 
 In Section \ref{section3}, 
 we prove a  {\em weak} dynamic programming principle for our mixed problem with ${\cal E}^f$-expectation.
 This  requires some specific techniques of stochastic analysis and BSDEs to handle measurability  and other issues   due to the nonlinearity of the expectation and the lack of regularity of the terminal reward. 
 Using the dynamic programming principle and properties of  RBSDEs, we prove  in Section  \ref{sec4} that the value function of our mixed problem is a {\em weak}  viscosity solution of 
a nonlinear  HJB variational inequality.  
In the Appendix, we give several fine measurability properties which are used in the paper. 

\section{Formulation of the mixed stopping/control problem}\label{sec2}

We consider the product space $\Omega:=\Omega_W \otimes \Omega_N$, where $\Omega_W:= {\cal C} ([0,T])$ is the  Wiener space, that is the set of continuous functions $\omega^1$ from $[0,T]$ into $\mathbb{R}^p$ such that $\omega^1(0) = 0$, and $\Omega_N := \mathbb{D}([0,T])$ is the Skorohod space of right-continuous with left limits (RCLL)  functions $\omega^2$ from $[0,T]$ into  $\mathbb{R}^d$,  such that 
 $\omega^2(0) = 0$.   
Recall that $\Omega$ is a Polish space for the topology of Skorohod. 
Here  $p, d \geq 1$, but, for notational simplicity, we shall consider only $\mathbb{R}$-valued functions, that is the case $p=d=1$. 

Let $B= (B^1, B^2)$ be the canonical process defined for each $t \in [0,T]$ and each $\omega= (\omega^1, \omega^2)$ by $B^i_t(\omega)= B^i_t(\omega^i):=\omega^i_t$, for $i=1,2$.   
Let us denote the first coordinate process $B^1$ by $W$. 
Let $P^W$ be the probability measure on  $(\Omega_W,\mathcal{B}(\Omega_W))$ such that $W$ is a Brownian motion. Here $\mathcal{B}(\Omega_W)$ denotes the Borelian $\sigma$-algebra on $\Omega_W$.

Set ${\bf E}:= \mathbb{R}^n \backslash \{0\}$ equipped with its Borelian $\sigma$-algebra $\mathcal{B}(\bf{E})$, where $n \geq 1$. 
We  define 
 the jump random measure $N$ as follows: for each $t>0$ and each ${\bf B}$ $ \in \mathcal{B}(\bf{E})$, 
\begin{equation}\label{measure}
N(., [0,t] \times {\bf B}):= \sum_{0< s \leq t} \textbf{1}_{\{\Delta B_s^2 \in {\bf B} \}}.
\end{equation}
The measurable set $({\bf E}, \mathcal{B}(\bf{E}))$ is equipped with a $\sigma$-finite positive measure $\nu$ such that  $\int_{\bf E} (1 \wedge |e| ) \nu(de) < \infty$.
Let $P^N$ be the probability measure on  $(\Omega_N,\mathcal{B}(\Omega_N))$ 
 such that $N$ is a Poisson random measure with compensator $\nu(de)dt$ and such that 
 $B_t^2= \sum_{0< s \leq t} \Delta B_s^2 $ a.s. Note that the sum of  jumps  is well defined up to a $P^N$-null set.
  We set  $\Tilde{N}(dr,de) := {N}(dr,de) - \nu(de)dt$. 
  The space $\Omega$ is equipped with the $\sigma$-algebra $\mathcal{B}(\Omega)$ and the probability measure $P:= P^W \otimes P^N$.
  Let  $\mathbb{F}:=(\mathcal{F}_t)_{t \geq 0}$ be  the filtration generated by $W$ and $N$ {\em completed with respect to $\mathcal{B}(\Omega)$ and $P$}, defined as follows (see \cite{J} p.3 or \cite{DM1} IV): 
let  ${\cal F}$ be the {\em completion} $\sigma$-algebra of $\mathcal{B}(\Omega)$ with respect to $P$ \footnote{For the definition of the 
{\em completion} of a 
$\sigma$-algebra and the one of  {\em $P$-null sets}, see e.g. Lemma \ref{Cra}.} For each $t \in [0,T]$,
 ${\cal F} _t$ is the $\sigma$-algebra generated by $W_s, N_s, s\leq t$ and the {\em $P$-null sets}. Note that ${\cal F}_T= {\cal F}$ and ${\cal F}_0$ is the $\sigma$-algebra generated by the {\em $P$-null sets}.
  Let $\mathcal{P}$ be the predictable $\sigma$-algebra on $\Omega \times [0,T]$ associated with the filtration $\mathbb{F}$.

 Let $T>0$ be fixed. Let $\mathbb{H}^{2}_T$ (denoted also by $\mathbb{H}^{2}$) be the set of real-valued predictable processes ($Z_t$) such that $\mathbb{E}\int_0^T Z_s^{2}ds < \infty$ and let
$\mathcal{S}^2$ be the set of real-valued RCLL adapted processes $(\varphi_s)$  with $\mathbb{E}[\sup_{0\leq s \leq T}\varphi_s^2] < \infty. $
Let ${L}_{\nu}^2$ be the set of measurable functions $l:({\bf E},\mathcal{B}(\bf{E})) \rightarrow (\mathbb{R}, {\cal B} (\mathbb{R}))$  such that  
$\| l \|_\nu^2 :=\int_{{\bf E}}l^2(e)\nu(de)< \infty.$
The set  ${L}^2_\nu$ is a 
 Hilbert space equipped with the scalar product 
$\langle l    , \, l' \rangle_\nu := \int_{{\bf E}} l(e) l'(e) \nu(de)$ for all $ l , \, l' \in {L}^2_\nu \times {L}^2_\nu.$ 
Let 
$\mathbb{H}^{2}_{\nu}$ denote the set of predictable real-valued processes $(k_t(\cdot))$ with $\mathbb{E}\int_0^T \|k_s\|_{{L}_{\nu}^2}^2ds < \infty$.

 Let   $\mathcal{A}$ be the set of controls,  defined as the set of predictable processes $\alpha$  valued in a compact subset $\bf A$ of $\mathbb{R}^p$, where $p \in 
 \mathbb{N}^*$.
For each $\alpha \in \mathcal{A}$
 and each initial condition $x$ in $\mathbb{R}$,  let $( X_s^{\alpha,x})_{ 0 \leq s \leq T}$ be the unique $\mathbb{R}$-valued solution in $\mathcal{S}^2$ of the stochastic differential equation (SDE):
\begin{equation}\label{richesse}
X_s^{\alpha,x}=\dd x+ \int_0^s b(X_r^{\alpha,x}, \alpha_r)dr+\int_0^s \sigma(X_r^{\alpha,x}, \alpha_r)dW_r+\int_0^s \int_{{\bf E}} \beta(X_{r^{-}}^{\alpha,x}, \alpha_r,e) \Tilde{N}(dr,de),
\end{equation}
where  $b, \ \sigma :\mathbb{R} \times {\bf A} \rightarrow \mathbb{R}$, 
are  Lipschitz continuous with respect to $x$ and $\alpha$,   and $\beta : \mathbb{R} \times {\bf A} \times {\bf E} \rightarrow \mathbb{R}$ is  a bounded measurable function such that for some constant $C \geq 0$, and for all $e \in \textbf{E}$
\begin{align*}
&|\beta(x,\alpha,e)| \leq C\,\Psi(e),
  \;\; x \in \mathbb{R}, \alpha \in {\bf A} \quad {\rm where} \quad  \Psi \in   {L}^2_\nu .\\
&|\beta(x,\alpha,e)- \beta(x',\alpha',e)| \leq C(|x-x'| + |\alpha - \alpha'|)
\Psi(e), 
\;\;  x, x'\in \mathbb{R}, \alpha, \alpha' \in {\bf A}.
\end{align*}

 The criterion of our mixed control problem, depending on $\alpha$, is defined via  a BSDE with driver function  $f$ satisfying the following hypothesis:
\begin{assumption} \label{H1}
%
%
$f: {\bf A} \times [0,T] \times  \mathbb{R}^3 \times  {L}_{\nu}^2  \rightarrow (\mathbb{R}, \mathcal{B}(\mathbb{R}))$ is  $ 
\mathcal{B}({\bf A}) \otimes \mathcal{B}([0,T]) \otimes \mathcal{B}(\mathbb{R}^3) \otimes \mathcal{B}({L}_{\nu}^2)$-measurable
and satisfies
\begin{itemize}
\item[(i)]
 $|f(\alpha,t,x,0,0,0)| \leq C(1+|x|^p), \forall \alpha \in {\bf A}, t \in [0,T],   x\in\mathbb{R}$, where  $p \in 
 \mathbb{N}^*$.

\item[(ii)]
$ |f(\alpha,t,x,y,z,k)- f (\alpha',t,x',y',z',k')| \leq C(  |\alpha -\alpha '|+ |x-x'|+|y-y'|+|z-z'|+\|k-k'\|_{{L}_{\nu}^2})$, $\forall t \in [0,T]$, $x, x', y, y', z,z' \in \mathbb{R}$, $k,k' \in {L}_{\nu}^2, \alpha, \alpha' \in {\bf A}.$

\item[(iii)]
$ f(\alpha,t,x,y,z,k_2)- f (\alpha,t,x,y,z,k_1) \geq <\gamma(\alpha,t,x,y,z,k_1,k_2),k_2-k_1>_{\nu}, \forall t,x,y,z,k_1,k_2, \alpha,$
\end{itemize}
where $\gamma: {\bf A} \times [0,T] \times  \mathbb{R}^3 \times  ({L}_{\nu}^2)^2  \rightarrow ({L}_{\nu}^2, \mathcal{B}({L}_{\nu}^2))$ is $ 
\mathcal{B}({\bf A}) \otimes \mathcal{B}([0,T]) \otimes \mathcal{B}(\mathbb{R}^3) \otimes \mathcal{B}(({L}_{\nu}^2)^2)$-measurable,  \\
satisfying $\gamma(.)(e)\geq -1$  and $ |\gamma(.)(e)| \leq  \Psi (e)\,$ $d\nu(e)$-a.s.\,, where  $\Psi \in   {L}^2_\nu$.
\end{assumption}

For all $x \in  \mathbb{R}$ and all control $\alpha \in \mathcal{A}$, let $f^{\alpha,x}$ be the driver defined by 
$$f^{\alpha,x}(r, \omega ,y,z,k):= f(\alpha_r(\omega),r, X_r^{\alpha,x}(\omega),y,z,k).$$

We introduce   the nonlinear  expectation $\mathcal{E}^{f^{\alpha,x}}$ 
(denoted more simply by $\mathcal{E}^{\alpha, x}$) associated with $f^{\alpha,x}$, defined for each stopping time $\tau$ and for each  $\eta \in {L}^2(\mathcal{F}_\tau)$ as:
$$\mathcal{E}^{\alpha,x}_{r,\tau}[\eta]:=\mathcal{X}_r^{\alpha,x},  \;\; 0 \leq r \leq \tau,$$
where $(\mathcal{X}_r^{\alpha,x})$ is the solution in $\mathcal{S}^2$ of the BSDE associated with driver $f^{\alpha,x}$, terminal time $\tau$ and terminal condition $\eta$, that is satisfying:
\begin{align*}
-d\mathcal{X}_r^{\alpha,x}=f(\alpha_r,r,X_r^{\alpha,x},\mathcal{X}_r^{\alpha,x},Z_r^{\alpha,x},K_r^{\alpha,x}(\cdot))dr-Z_r^{\alpha,x}dW_r-\int_{{\bf E}}K_r^{\alpha,x}(e)\Tilde{N}(dr,de); \quad
\mathcal{X}_S^{\alpha,x}=\eta, 
\end{align*}
and $(Z_s^{\alpha,x})$, $(K_s^{\alpha,x})$ are the associated processes, which belong respectively to $\mathbb{H}^{2}$ and  $\mathbb{H}^{2}_{\nu}$. 
Condition (iii) ensures the non decreasing property of the $\mathcal{E}^{f^{\alpha,x}}$-expectation
 (see \cite{16}).

%

For all $x \in  \mathbb{R}$ and all control $\alpha \in \mathcal{A}$, we define the reward  by
$h(s,X_s^{\alpha,x})$ for $0 \leq s < T$ and $g(  X_T^{\alpha,x})$ for $s=T$, where 

%
\begin{itemize}
\item[$\bullet$]
$g :  \mathbb{R} \rightarrow \mathbb{R}$ is Borelian.

%
%
%
%
%
\item[$\bullet$]
 $h :[0,T] \times \mathbb{R} \rightarrow \mathbb{R}$ is a function which
is Lipschitz continuous with respect to $x$ uniformly in $t$, and continuous with respect to $t$ on $[0,T]$.
\item[$\bullet$]
 $|h(t,x)| +  |g(x)| \leq C(1+ |x|^p), \forall t \in [0,T], x \in \mathbb{R},  \text{ with $p\in \mathbb{N}^*$}. $
\end{itemize}

 Let $\mathcal{T}$ be the set of stopping times with values in $[0,T]$. 
Suppose the initial time is equal to $0$. Note that $ \mathcal{E}_{0, \tau}^{\alpha,t,x} 
  [ \bar h(\tau, X_{\tau}^{\alpha,t,x})]$ can be taken as constant. 
  \footnote{
Indeed, the solution of a BSDE with Lipschitz driver is unique up to a $P$-null set. Its initial value may thus be taken  constant for all $\omega$, modulo a change of its value on a $P$-null set, because ${\cal F}_0$ is the $\sigma$-algebra generated by the $P$-null sets.}
For each initial condition $x \in \mathbb{R}$, we consider the  mixed optimal control/stopping problem:
\begin{equation}\label{probform}
u(0,x):= \sup_{\alpha \in \mathcal{A}} \sup_{\tau \in \mathcal{T}} \mathcal{E}_{0, \tau}^{\alpha,x}[\bar h(\tau, X_{\tau}^{\alpha,x})],
\end{equation}
where  $$\bar h (t,x): = h (t,x){\bf 1}_{t < T} + g(x){\bf 1}_{t = T}.$$ Note that $\bar h$ is Borelian but not necessarily regular in $(t,x)$.
 
We now make the problem dynamic. We  define, for $t \in [0,T]$ and each $\omega$ $\in$ $\Omega$ the $t$-translated path  $\omega^t = (\omega^t_s)_{s \geq t}:=(\omega_s-\omega_t)_{s \geq t}$.
Note that $(\omega^{1,t}_s)_{s \geq t}:=(\omega_s^1-\omega_t^1)_{s \geq t}$  corresponds to the realizations of the translated Brownian motion $W^t:= (W_s-W_t)_{s \geq t}$ and that the translated Poisson random measure $N ^t :=
N(]t,s],.)_{s \geq t}$ can be expressed in terms of $(\omega^{2,t}_s)_{s \geq t}:=(\omega_s^2-\omega_t^2)_{s\geq t}$
 similarly to \eqref{measure}.  Let $\mathbb{F}^t=(\mathcal{F}_s^t)_{t \leq s \leq T}$ be the  filtration 
generated by $W^t$ and $N ^t$ {\em completed with respect to ${\cal B}(\Omega)$ and $P$}. 
Note that for each $s\in [t,T]$, $\mathcal{F}_s^t$ is the $\sigma$-algebra generated by $W_r^t$, $N_r ^t$, $t\leq r\leq s$  and $\mathcal{F}_0$. Recall also that  we have a martingale representation theorem for 
  ${\mathbb F}^t$-martingales as stochastic integrals with respect to $ W^t$ and $\Tilde{N}^t$.
 
Let us denote by $\mathcal{T}^t_{t}$ the set of stopping times with respect to $\mathbb{F}^t$ with values in $[t,T]$.  Let   $\mathcal{P}^t$ be the predictable $\sigma$-algebra on 
$\Omega \times [t,T]$ equipped with the filtration $\mathbb{F}^t$.\\
We now introduce the following spaces of processes. Let $t \in [0,T]$.
 Let $\mathbb{H}_t^2$ be  the ${\cal P}^t$-measurable processes $Z$  on $\Omega \times [t,T]$ such that 
 $\| Z \|_{\mathbb{H}_t^2} := \mathbb{E}[\int_t^T Z_u^2du] < \infty.$ We define  
 $\mathbb{H}_{t, \nu}^2$ as  the set of ${\cal P}^t$-measurable processes $K$  on $\Omega \times [t,T]$ such that 
 $\|K \|_{\mathbb{H}_{t, \nu}^2} := \mathbb{E}[\int_t^T  ||K_u||_{\nu}^2du] < \infty.$
 We denote by ${\cal S}^2_t$ the set of real-valued RCLL processes $\varphi$ on $\Omega \times [t,T]$, 
 ${\mathbb F}^t$-adapted, with $\mathbb{E}[\sup_{t\leq s \leq T} \varphi_s^2] < \infty. $

 Let   $\mathcal{A}^t_t$ be the set of controls $\alpha:\Omega \times [t, T] \mapsto {\bf A}$, which 
 are ${\cal P}^t$-measurable. We consider the solution denoted by $X^{\alpha,t,x}$ in ${\cal S}^2_t $ 
 of the following SDE driven by the translated Brownian motion $W^t$ and the translated Poisson random measure $N^t$ (with filtration $\mathbb{F}^t$) : 
 \begin{equation}\label{richessebis}
X_s^{\alpha,t,x}=\dd x+ \int_t^s b(X_r^{\alpha,t,x}, \alpha_r)dr+\int_t^s \sigma(X_r^{\alpha,t,x}, \alpha_r)dW^t_r+\int_t^s \int_{{\bf E}} \beta(X_{r^{-}}^{\alpha,t,x}, \alpha_r,e) \Tilde{N}^t(dr,de).
\end{equation}
For all $(t,x) \in [0,T] \times \mathbb{R}$ and all control $\alpha \in \mathcal{A}_t ^t$, let $f^{\alpha,t,x}$ be the driver defined by 
$$f^{\alpha,t,x}(r, \omega ,y,z,k):= f(\alpha_r(\omega),r, X_r^{\alpha,t,x}(\omega),y,z,k).$$
   Let  $\mathcal{E}_{., \tau}^{\alpha,t,x}  [ {\bar h}(\tau, X_{\tau}^{\alpha,t,x})]$ (denoted also by $ \mathcal{X}_\cdot^{\alpha,t,x}$)
 be  the solution in ${\cal S}^2_t $ of the BSDE with driver $f^{\alpha,t,x}$, terminal time $\tau$ and terminal condition $\bar h(\tau, X_{\tau}^{\alpha,t,x})$, driven by $W^t$ and $N^t$, which is solved on $[t,T]\times \Omega$ with respect to the filtration $\mathbb{F}^t$: 
\begin{align}\label{numero}
\begin{cases}
-d\mathcal{X}_r^{\alpha,t,x}=f(\alpha_r,r,X_r^{\alpha,t,x},\mathcal{X}_r^{\alpha,t,x},Z_r^{\alpha,t,x},K_r^{\alpha,t,x})dr-Z_r^{\alpha,t,x}dW^t_r-\int_{{\bf E}}K_r^{\alpha,t,x}(e)\Tilde{N}^t(dr,de)\\
\mathcal{X}_\tau^{\alpha,t,x}=\bar h(\tau, X_{\tau}^{\alpha,t,x}),
\end{cases}
\end{align}
where $Z_\cdot^{\alpha,t,x}$, $K_\cdot^{\alpha,t,x}$ are the associated processes, which belong respectively to $\mathbb{H}_t^{2}$ and  $\mathbb{H}^{2}_{t,\nu}$. Note that $ \mathcal{E}_{t, \tau}^{\alpha,t,x} 
  [ \bar h(\tau, X_{\tau}^{\alpha,t,x})]$ can be taken deterministic modulo a change of its value on a $P$-null set. 
\footnote{Indeed,  the solution $ \mathcal{E}_{., \tau}^{\alpha,t,x} 
  [ \bar h(\tau, X_{\tau}^{\alpha,t,x})] (= \mathcal{X}_.^{\alpha,t,x})$ of the BSDE \eqref{numero} is unique up to a $P$-{\em null set}. Moreover, its value at time $t$ is ${\cal F}^t_t$-measurable, and ${\cal F}^t_t$ is equal to  the $\sigma$-algebra generated by the $P$-{\em null sets} (that is ${\cal F}_0$). The same property holds for the solution of the reflected BSDE \eqref{4.4}.
 See also the additional remarks  \ref{tb} and \ref{remarque}.
}

For each initial time $t$ and each initial condition $x$, we define the value function as
\begin{equation}\label{probform1}
u(t,x):= \sup_{\alpha \in \mathcal{A}^t_t} \sup_{\tau \in \mathcal{T}^t_{t}}  \mathcal{E}_{t, \tau}^{\alpha,t,x} 
  [ \bar h(\tau, X_{\tau}^{\alpha,t,x})],  
\end{equation}
which is  a  deterministic function of $t$ and $x$. 

For each $\alpha \in  \mathcal{A}^t_t$, we introduce the function $u^{\alpha}$ defined as 
 $$u^{\alpha}(t,x):=\sup_{\tau\in \mathcal{T}^t_{t}} \mathcal{E}_{t, \tau}^{\alpha,t,x}
   [\bar h(\tau, X_{\tau}^{\alpha,t,x})].$$
 We thus get 
 \begin{equation}\label{eg}
 u(t,x) =\sup_{\alpha \in \mathcal{A}^t_t} u^{\alpha}(t,x).
 \end{equation}
 For each $\alpha$, $u^{\alpha}(t,x) \geq \bar h(t,x)$,  and hence $u(t,x) \geq \bar h(t,x)$. 
 Moreover, $u^{\alpha}(T,x) = u(T,x)= g(x)$. 
 
By Theorem 3.2 in $\cite{17}$, for each $\alpha$, the value function $u^\alpha$ is related to a reflected BSDE. 
More precisely, let $(Y^{\alpha,t,x},  Z^{\alpha,t,x},   K^{\alpha,t,x})  \in \mathcal{S}_t^2 \times \mathbb{H}_t^2 \times
\mathbb{H}^2_{\nu, t}$ be the solution of the reflected BSDE
 associated 
with driver $f^{\alpha,t,x}:= f(\alpha_\cdot,\cdot,X_\cdot^{\alpha,t,x},y,z,k)$,  (RCLL) obstacle process $\xi^{\alpha,t,x}_s:= {\bar h}(s, X_{s}^{\alpha,t,x})_{t \leq s \leq T}$, terminal condition 
$g(X^{\alpha,t,x}_T)$, and with filtration $\mathbb{F}^t$, that is
\begin{equation}\label{4.4}
\begin{cases}
\!-dY^{\alpha,t,x}_r \!=f(\alpha_r,r,X^{\alpha,t,x}_r,Y^{\alpha,t,x}_r,Z^{\alpha,t,x}_r, K^{\alpha,t,x}_r)dr+ dA^{\alpha,t,x}_s- Z^{\alpha,t,x}_rdW^t_r-\int_{{\bf E}}K^{\alpha,t,x}(r,e) \Tilde{N}^t(dr,de),\\

\,\,Y^{\alpha,t,x}_T\,=\,g(X^{\alpha,t,x}_T)\,\text{ and } \,Y^{\alpha,t,x}_s \geq \, \,\xi^{\alpha,t,x}_s= h(s, X_{s}^{\alpha,t,x}), \,\,0 \leq s < T\; \text{ a.s. },\\
A^{\alpha,t,x} \text{ is a RCLL nondecreasing } \,{\cal P}^t\,  \text{-measurable process with } A^{\alpha,t,x}_t=0 \text{ and such that }\\
\int_0^T (Y^{\alpha,t,x}_s - \xi^{\alpha,t,x}_s) dA^{\alpha,t,x,c}_s = 0 \text{ a.s. and } \;  \Delta A^{\alpha,t,x,d}_{s}= - \Delta A^{\alpha,t,x}_s \,{\bf 1}_{\{Y^{\alpha,t,x}_{s^-} = \xi^{\alpha,t,x}_{s^-}\}}  \quad   \rm{a.s.} 
\end{cases}
\end{equation}
Here  $A^{\alpha,t,x,c}$ denotes the continuous part of $A$ and $A^{\alpha,t,x,d}$  its discontinuous part.
In the particular case when $h(T,x) \leq g(x)$, then the obstacle $\xi^{\alpha,t,x}$ satisfies for all $\mathbb{F}^t$-predictable stopping time $\tau$, $\xi_{\tau^-} \leq \xi_{\tau}$ a.s. 
which implies the continuity 
of the process $A^{\alpha,t,x}$ (see \cite{17}).


In the following, for each $\alpha \in \mathcal{A}_t^t$,  
$Y^{\alpha,t,x}_\cdot$ will be also denoted by  ${Y}_{\cdot,T}^{\alpha,t,x}[g(X^{\alpha,t,x}_T)]$. Note that  its value at time $t$ can be taken as deterministic   modulo a change of its value on a $P$-null set.

Using Theorem 3.2 in $\cite{17}$, we  get that for each $\alpha \in \mathcal{A}_t^t$,
\begin{equation}\label{caractb}
u^\alpha(t,x)=Y_t^{\alpha,t,x}=  {Y}_{t,T}^{\alpha,t,x}[g(X^{\alpha,t,x}_T)].
\end{equation}


\noindent By using these equalities, we can reduce our mixed optimal stopping/control problem \eqref{probform1} to an optimal  control problem for reflected BSDEs: 
\begin{theorem}[Characterization of the value function]\label{representation}
For each $(t,x) \in [0,T] \times {\mathbb R}$, the value function $u(t,x)$ of the mixed optimal stopping/ control problem  \eqref{probform1} satisfies
\begin{equation}\label{representationbis}
u(t,x) = \sup_{\alpha \in \mathcal{A}^t_t} u^\alpha(t,x)= \sup_{\alpha \in \mathcal{A}^t_t} {Y}_{t,T}^{\alpha,t,x}[g(X^{\alpha,t,x}_T)].
\end{equation}
\end{theorem}

This key property will be used to solve our mixed problem. We point out that in the classical case of linear expectations, this approach allows us to provide alternative proofs of the dynamic programming principle to those given in the previous literature.


%
%
%
\begin{remark} 
Some mixed optimal control/stopping problems with nonlinear expectations have been studied in \cite{BY11, 17}. In these papers,  the reward process does not depend on the control, which yields the characterization of 
 the value function    as the  solution of an RBSDE. This is not the case here. 
%
\end{remark}
\section{Weak Dynamic Programming Principle}\label{section3}
In this section, we  prove a {\em weak} dynamic programming principle for our mixed optimal control/stopping problem \eqref{probform1}. 
To this purpose, we first provide some splitting properties  for the forward-backward system \eqref{richessebis}-\eqref{4.4}. We then show some measurability properties of the function  
$u^{\alpha}(t,x)$, defined by \eqref{caractb}, with respect to both the state variable $x$ and the control $\alpha$. Using these results, 
we show the existence of ${\varepsilon}$-optimal controls satisfying some appropriate measurability properties.
Moreover, we  establish  a Fatou lemma  for RBSDEs, where the limit involves both terminal condition and terminal time. Using these results, we then prove a sub- (resp. super-) optimality principle of dynamic programming, involving the u.s.c. (resp. l.s.c.)  envelope of the value function.
\subsection{Splitting properties}


   Let $s \in [0,T]$.
For each $\omega$, let $^s\omega:= (\omega_{r \wedge s})_{0 \leq r \leq T}$ and 
$\omega^s := (\omega_{r} - \omega_{s})_{s \leq r \leq T}$.\\
We shall identify the path $\omega$ 
with $ (^s\omega, \omega^s),$
which means that a path can be splitted into two parts: the path before time $s$ and the 
$s$-translated path after time $s$. \\
Let $\alpha$ be a given control in $\mathcal{A}$. We show below the following:  at time $s$, for fixed past path  $\tilde \omega:=$$^s  \omega$, the process $\alpha( \tilde \omega, .)$ which only depends on the future path $\omega^s$ is an
 $s$-admissible control, that is 
$\alpha( \tilde \omega, .) \in \mathcal{A}_{s}^{s}$; furthermore,
 the criterium 
 $Y^{\alpha,0,x} (\tilde \omega, .)$ from time $s$ coincides with the solution of the reflected BSDE driven by $ W^s$ and $\Tilde{N}^s$,  controlled by $\alpha( \tilde \omega, .)$
 and associated with initial time $s$ and initial state condition $X_s^{\alpha,0,x}(\tilde \omega)$. 
 
 We introduce the following random variables defined on $\Omega$ by 
$$ S^s: \omega  \mapsto  \,^s\omega \,\,;\quad T^s: \omega \mapsto  \, \omega^s.$$
Note that they are independent. 
For each $ \omega \in \Omega$, we have
$\omega = S^s(\omega) + T^s(\omega) {\bf 1}_{]s,T]}, $
or equivalently
$\omega_r = \omega_{r \wedge s} + \omega_r^s {\bf 1}_{]s,T]}(r), $ fort all $ r \in [0,T].$

%

%

For all paths $\omega, \omega '$ $\in \Omega$, $(^s  \omega,  T^s(\omega '))$ denotes 
the path such that the past trajectory before  $s$ is that of $\omega$, and the $s$-translated trajectory after $s$ is that of $\omega '$. This can also be written as:
$(^s  \omega,  T^s(\omega ')) :=\,\, ^s  \omega + T^s(\omega '){\bf 1}_{]s,T]}.$
Note that for each $\omega$ $\in \Omega$, we have $(^s  \omega, T^s(\omega ))= \omega$. 

\begin{lemma}\label{measurability}
Let $s \in [0,T]$. 
Let $Z \in\mathbb{H}^2$. 
There exists a $P$-null set ${\cal N}$ such that for each $ \omega$ in the complement   ${\cal N}^c$ of ${\cal N}$, 
 setting  $\tilde \omega:= $ $^{s}\omega = \omega_{. \wedge s} $, the process $Z(\tilde \omega, T^s)$ (denoted also by $Z(\tilde \omega, .)$) defined by
$$Z(\tilde \omega, T^s): \Omega  \times [s,T] \rightarrow {\mathbb R} \,; \,(\omega ', r) \mapsto Z_r(\tilde \omega, T^s(\omega '))$$
belongs to  
$\mathbb{H}_s^2$.
Moreover, if $Z \in$ $\mathcal{A}$, then $Z(\tilde \omega, T^s)  \in \mathcal{A}_{s}^{s}.$\\
This property also holds for all initial time $t \in [0,T]$.
More precisely, let $s \in [t,T]$. 
Let $Z \in\mathbb{H}_t^2$ (resp. $\mathcal{A}_{t}^{t}$).
For a.e. $\omega$ $\in$ $\Omega$,  the process 
$Z(^{s}\omega, .) = (Z_r(^{s}\omega, T^s))_{r \geq s}$
belongs to  
$\mathbb{H}_s^2$ (resp. $\mathcal{A}_{s}^{s}$).

\end{lemma}

\dproof 
Classically, we  have 
$\mathbb{E}[ \int_{s}^T Z_r^2 dr ]$$=$ $\mathbb{E}[\mathbb{E}[ \int_{s}^T Z_r^2 dr |  \cal{F} _{s} ]]$ $ < + \infty$. 
\noindent Using the independence of $T^{s}$ with respect to ${\cal F}_{s}$ and the measurability of $ S^{s}$ with respect to ${\cal F}_{s}$, we derive that 
$$\mathbb{E}[ \int_{s}^T Z_r^2 dr | \,  \cal{F}_{s} ]= \mathbb{E}[ \int_{s}^T Z_r (S^s , T^s)^2 dr | \,  \cal{F}_{s} ]= F( \, S^{s})
< + \infty \quad P-{\rm a.s.}\,, $$ where 
$F( \tilde \omega ) := \mathbb{E}[ \int_{s}^T Z_r(\tilde \omega, T^{s} (\cdot))^2 dr] .$

Let us now prove that the process $\,Z (\tilde \omega, T^s)\,: \,(\omega', r) \mapsto Z_r(\tilde \omega, T^s(\omega'))$ is ${\cal P}^s$-measurable.
There exists a process  indistinguishable of $(Z_r)$, still denoted by $(Z_r)$, which is measurable with respect to the predictable $\sigma$-algebra associated with the filtration generated by
 $W$ and $N$ (see \cite{DM1}   IV \textsection 79). We can thus
 suppose in this proof (without loss of generality)  that ${\mathbb F}$ (resp. ${\mathbb F}^{s}$) is the filtration generated by $W$ and $N$ (resp. $W^s$ and $N^s$)
, and $\mathcal{P}$ (resp. $\mathcal{P}^s$) is its associated predictable $\sigma$-algebra.
 Suppose we have shown that the map 
$\psi: \Omega \times [s,T] \rightarrow \Omega \times [0,T]$; $
(\omega', r) \mapsto \left((\tilde \omega, T^s(\omega')),r \right)$
is $({\cal P}^{s}, {\cal P})$-measurable.
Now, we have $Z (\tilde \omega, T^s) (\omega', r)
= Z \circ \psi (\omega', r)$ for each $(\omega', r)$ $\in$ $\Omega \times [s,T]$.
Since 
 $Z$ is $\mathcal{P}$-measurable, by composition, we derive that $\,Z (\tilde \omega, T^s)$ is 
 ${\cal P}^{s}$-measurable.

It remains to show the $({\cal P}^{s}, {\cal P})$-measurability of $\psi$.
 Recall that the $\sigma$-algebra ${\cal P}$ is generated by the sets 
$H \times ]v,T]$, where $v \in [0,T[$ and $H$ is 
 of the form: $H = \{B_{t_i} \in A_i, \,\,\, 1\leq i \leq n
 \},$ where $ A_i \in \mathcal{B}(\mathbb{R}^2)$ and $t_1< t_2 < ... \leq v $. It is thus sufficient to show that 
$\psi ^{-1}(H \times ]v,T])$ $\in$ ${\cal P}^{s}$. 
Note  that 
$\psi ^{-1}(H \times ]v,T])= H' \times ]v,T]$, where 
$H' = \{ \omega' \in \Omega\,,\, (\tilde \omega, T^s(\omega')) \in H\}$. 
If there exists $i$ such that $t_i \leq s$  and $\tilde{\omega}_{t_i} \not \in A_i$, then  $H'=
 \emptyset$. Otherwise,
we have 
$H'=  \{ \omega'_{t_i}- \omega'_s \in A_i, \,\,\, \forall \,\,i \text{ such that } t_i > s  \}$.
Hence $H' \in {\cal F}_v^{s}$, which implies that $\psi ^{-1}(H \times ]v,T])$ $\in$ ${\cal P}^{s}$. The proof is thus 
complete.
\fproof

Let $Z\in \mathbb{H}^2$. Let us give an intermediary time $s \in [0,T]$ and a fixed past path 
${^s} \omega$. Note that the Lebesgue integral $(\int_s^{u} {Z}_r dr)({^s} \omega, .)$ 
is equal a.s. to the integral $\int_s^{u} {Z}_r({^s} \omega, .) dr$. 
We now show that the stochastic integral
$(\int_s^{u} {Z}_r dW_r)({^s} \omega, .) $ coincides with the  stochastic integral  
of the process $Z({^s} \omega, .) $ with respect to the translated Brownian motion $W^s$, that is 
 $\int_s^{u} {Z}_r({^s} \omega, .) dW_r^s$. 
\begin{lemma}\label{zz}(Splitting properties for stochastic integrals)
Let $s \in [0,T]$. 
Let $Z\in \mathbb{H}^2$ and $K \in \mathbb{H}_{\nu}^2$.
There exists a $P$-null set $\cal N$ (which depends on $s$) such that for each $ \omega \in {\cal N}^c$, and $\tilde \omega:= $ $^{s}\omega$, we have
$(Z_r(\tilde \omega, T^s))_{r\geq s}$ $\in$ 
$\mathbb{H}_{s}^2$ and $(K_r(\tilde \omega, T^s))_{r\geq s}$ $\in$ 
$\mathbb{H}_{s, \nu}^2$,
and 
\begin{eqnarray}\label{ito}
 (\int_s^{u} {Z}_r dW_r) 
 (\tilde \omega , T^s)  &=  &  \int_s^{u} {Z}_r(\tilde \omega, T^s) dW_r^s \quad P-{\rm a.s.}\\
(\int_s^u \int_{{\bf E}} K_r(  e) \Tilde{N}(dr,de)) (\tilde \omega , T^s)  &= &
   \int_s^u \int_{{\bf E}} K_r(\tilde \omega , T^s,  e) \Tilde{N}^s(dr,de)  \quad P-{\rm a.s.}\nonumber.
\end{eqnarray}
  \end{lemma}
  
  \begin{remark}\label{tb}
In the literature, the $s$-translated Brownian motion is   often defined by 
$W'_v:= W_{s+v}- W_s= W^s_{s+v}$, $0 \leq v \leq T-s$. 
For each $Z$ $\in$ $\mathbb{H}_s^2$ and for each $u\geq s$, we have  
$\int_s^{u} {Z_r} dW_r^s= \int_0^{u- s} Z_{s+r} dW'_r $ a.s. The use of $W^s$ thus allows us to avoid a change of time.
The same remark holds for the Poisson random measure. 

Note that equality \eqref{ito}  is equivalent to 
$
(\int_s^{u} {Z}_r dW_r) (\tilde \omega , T^s({\omega'}))= (\int_s^{u} {Z_r}(\tilde \omega, T^s) dW_r^s)(\omega')
$
for $P$-almost every $\omega'$ $\in \Omega$. The same remark holds for the second equality.

%
\end{remark}

 \dproof We shall only prove the first equality with the Brownian motion. 
 The second one with the Poisson random measure can be shown by similar  arguments. \\
  Let us first show that equality \eqref{ito} holds for a simple process. 
 Let $a <  T$ and let $H$ $\in$ ${\mathbb L}^2 ({\cal F}_a)$. For each 
 $\omega \equiv (^s$$\omega, \omega^s)= (S^s(\omega), T^s(\omega))$ $\in$ $\Omega$, we have
 $$ (\int_s^{u} H \textbf{1}_{]a,T]}dW_r)({^s}\omega, \omega^s)=H({^s}\omega, \omega^s)(\omega_{u}^s-\omega_{a \wedge u }^s)=(\int_s^{u}H({^s}\omega, T^s)\textbf{1}_{]a,T]}dW_r^s)(\omega) .$$
 Let now $Z\in \mathbb{H}^2$. Let us show that $Z$ satisfies equality \eqref{ito}. The idea is to approximate $Z$ by an appropriate sequence 
 of simple processes $(Z^n)_{n \in \mathbb N}$ so that the sequence $(Z^n)_{n \in \mathbb N}$ converges in $\mathbb{H}^2$ to $Z$, and that, for almost every past path $
^s \omega$, the sequence $(Z^n(^s \omega, T^s))_{n \in \mathbb N}$ converges to $Z(^s \omega, T^s)$ in $\mathbb{H}^2_s$.
 For each $n \in \mathbb{N}^*$, define 
 $$Z^n_r:=n \sum_{i=1}^{n-1}( \int_{\frac{(i-1)T}{n}}^{\frac{iT}{n}}Z_u du) \textbf{1}_{]\frac{i T}{n}, \frac{(i+1)T}{n}]}(r).$$ 
  By  inequality \eqref{Pn} in the Appendix, we have $\int_s^u (Z_r^n (\omega))^2dr \leq \int_s^u Z_r(\omega)^2 dr$, and for each $\omega \in \Omega$ and $s \leq u$, $\int_s^u (Z_r^n(\omega)-Z_r(\omega))^2 dr \rightarrow 0.$ Since 
  $\int_s^u Z_r ^2 dr$ $\in$ $L^1(\Omega)$, it follows, by the Lebesgue theorem for the conditional expectation, that 
 \begin{equation}\label{lim}
  \mathbb{E}[\int_{s}^{u} (Z_r^n-Z_r)^2 dr | {\cal F}_{s} ]\rightarrow 0 
  \end{equation}
   excepted on a $P$-null set ${\cal N}$. Since $S^s$ is ${\cal F}_{s}$-measurable and $T^s$ is independant of ${\cal F}_{s}$, there exists a $P$-null set included in the previous one, such that
 for each $ \omega \in {\cal N}^c$,  setting $\tilde \omega= $ $^{s}\omega$, we have
 \begin{eqnarray}\label{iso}
   \mathbb{E}[\int_{s}^{u} (Z_r^n-Z_r)^2 dr | {\cal F}_{s} ](\tilde \omega)&= & \mathbb{E}[ \int_s^{u}(Z_r^n(\tilde \omega, T^s) -Z_r(\tilde \omega,T^s))^2dr ]\nonumber\\
  &=&  \mathbb{E}[ ( \int_s^{u}Z_r^n(\tilde \omega, T^s)dW_r^s-\int_s^{u}Z_r(\tilde \omega,T^s)dW_r^s )^2].
  \end{eqnarray}
 The second equality follows by the classical isometry property. Now, for each square integrable martingale $M$, $M^2 - \langle M \rangle$ is a martingale. Hence, for each $ \omega \in {\cal N}^c$, where ${\cal N}$ is a $P$-null set included in the previous one, setting $\tilde \omega= $ $^{s}\omega$, we have
\begin{eqnarray}\label{iso2}
  \mathbb{E}[\int_{s}^{u} (Z_r^n-Z_r)^2 dr | {\cal F}_{s} ](\tilde \omega)
&=&  \mathbb{E}[(\int_{s}^{u} {Z}_r^n dW_r  - \int_{s}^{u} {Z}_r dW_r)^2 | {\cal F}_{s} ](\tilde \omega)\nonumber\\
&=&  \mathbb{E}[\left( (\int_{s}^{u} {Z}_r^n dW_r )(\tilde \omega, T^s) -( \int_{s}^{u} {Z}_r dW_r)(\tilde \omega, T^s) \right)^2].
\end{eqnarray}
For each $n \in \mathbb{N}^*$, since $Z^n$ is a simple process, it satisfies equality 
 \eqref{ito} everywhere,
  that is\\
$
(\int_s^{u} {Z}_r^n dW_r) (\tilde \omega , T^s) = \int_s^{u} {Z_r^n}(\tilde \omega, T^s) dW_r^s. 
$
By the convergence property \eqref{lim}, equalities \eqref{iso} and \eqref{iso2}, and the uniqueness property of the limit in $L^2$, we derive that for each $ \omega \in {\cal N}^c$, setting $\tilde \omega= $ $^{s}\omega$, equality \eqref{ito} holds. The proof is thus complete.
 \fproof 

Using the above lemmas, we now show that for each $s \geq t$, for almost every $\omega \in \Omega$, setting $\tilde \omega= $ $^{s}\omega$, the process 
 $Y^{\alpha,t,x} (\tilde \omega, T^s)$ coincides with the solution of the reflected BSDE 
 on $\Omega \times[s,T]$, associated with driver $f^{\alpha(\tilde \omega,T^s),s,\eta(\tilde \omega)}$, with obstacle
 $\bar h(r, X_r^{\alpha(\tilde \omega,T^s),s,X_s^{\alpha(\tilde \omega),t,x}(\tilde \omega)})$ and 
  filtration ${\mathbb F}^s$, and driven by $ W^s$ and $\Tilde{N}^s$.
     
 To simplify notation, $T^s$ will be replaced by $\cdot$ in the following. In particular $Y^{\alpha,t,x} (\tilde \omega, T^s)$ will be simply denoted by $Y^{\alpha,t,x} (\tilde \omega, .)$.
 
\begin{theorem}\label{egal} (Splitting properties for  the forward-backward ``system")
Let $t \in [0,T]$, $\alpha$ $\in$ ${\cal A}_t^t$ and
$s \in [t,T]$. 
There exists a $P$-null set ${\cal N}$ (which depends on $t$ and $s$) such that for each $ \omega \in {\cal N}^c$, setting $\tilde \omega= $ $^{s}\omega$, the following properties hold:
\begin{itemize}
\item
There exists an unique solution $(X_r^{\alpha(\tilde \omega,\cdot),s,\eta(\tilde \omega)}) _{ s \leq r \leq T}$
 in ${\cal S}^2_s$ of the following SDE:
\begin{eqnarray}\label{fofo}
X_r^{\alpha(\tilde \omega,\cdot),s,\eta(\tilde \omega)}\,=\, \,\dd \eta(\tilde \omega)\!\!\!&+ &\!\!\! \!\int_s^r b(X_v^{\alpha(\tilde \omega,\cdot),s,\eta(\tilde \omega)}, \alpha_v(\tilde \omega,\cdot))dv+\int_s^r \sigma(X_v^{\alpha(\tilde \omega,\cdot),s,\eta(\tilde \omega)}, \alpha_v(\tilde \omega,\cdot))dW^s_v \nonumber\\
&+&\int_s^r \int_{{\bf E}} \beta(X_{v^{-}}^{\alpha(\tilde \omega,\cdot),s,\eta(\tilde \omega)}, \alpha_v(\tilde \omega,\cdot),e) \Tilde{N}^s(dv,de),\,\, 
\end{eqnarray}
where $\eta(\tilde \omega):=X_s^{\alpha(\tilde \omega),t,x}(\tilde \omega)$.
 We also have
 $X_r^{\alpha, t,x}(\tilde \omega, .)  = X_r^{\alpha(\tilde \omega,\cdot),s,\eta(\tilde \omega)}$,
 $s\leq r \leq T$ $P$-a.s.
 \item
There exists an unique solution
  $(Y_{r}^{ \alpha(\tilde \omega, \cdot),s,\eta(\tilde\omega)}, Z_{r}^{ \alpha(\tilde\omega, \cdot),s,\eta(\tilde\omega)}, K_{r}^{ \alpha(\tilde\omega, \cdot),s,\eta(\tilde\omega)})
 _{ s \leq r \leq T}$ in 
${\cal S}^2_s$ $\times$ $\mathbb{H}^2_s$ $ \times$ $\mathbb{H}^2_{s, \nu}$ of the reflected BSDE on $\Omega \times[s,T]$ driven by $ W^s$ and $\Tilde{N}^s$ and associated with filtration $\mathbb{F}^s$, driver 
$f^{\alpha(\tilde \omega,.),s,\eta(\tilde \omega)}$,  and obstacle
 $\bar h(r, X_r^{\alpha(\tilde \omega,\cdot),s,\eta(\tilde \omega)})$. We have:
 \begin{eqnarray}\label{abc}
 Y_r^{\alpha,t,x} (\tilde\omega, .) & =& Y_{r}^{ \alpha(\tilde\omega, .),s,\eta(\tilde\omega)} ,
 \quad  s\leq r \leq T, \quad P-{\rm a.s.}\\
  Z_r^{\alpha,t,x} (\tilde\omega, .)   = Z_{r}^{ \alpha(\tilde\omega, \cdot),s,\eta(\tilde\omega)}& {\rm and}&
  K_r^{\alpha,t,x} (\tilde\omega,.)   = 
 K_{r}^{ \alpha(\tilde\omega, \cdot),s,\eta(\tilde\omega)} , \quad  s\leq r \leq T, \quad dP\otimes dr-{\rm a.s.} \nonumber
 \end{eqnarray}
\begin{equation}\label{abcd}
Y_s^{\alpha,t,x} (\tilde \omega, .)  = Y_{s}^{ \alpha(\tilde\omega, \cdot),s,\eta(\tilde\omega)} = u^{ \alpha(\tilde\omega, \cdot)}(s,\eta(\tilde\omega))  \quad P-{\rm a.s.}
\end{equation}
\end{itemize}
\end{theorem}

\dproof 
Recall that by Lemma \ref{measurability}, there exists a $P$-null set ${\cal N}$ such that for each $ \omega \in {\cal N}^c$, the process $\alpha(^{s}\omega, \cdot):= (\alpha_r(^{s} \omega, T_s))_{r\geq s}$ 
belongs to  $\mathcal{A}_{s}^{s}$.\\
 Let us show the first assertion. To simplify the exposition, we suppose that there is no Poisson random measure. There exists a $P$-null set, still denoted by ${\cal N}$, included in the above one such that for each $ \omega \in {\cal N}^c$, setting $\tilde \omega= $ $^{s}\omega$,
  \begin{eqnarray*}
X_r^{\alpha,t,x}(\tilde \omega, .) \,=\, \,\dd \eta(\tilde \omega)\!\!\!&+ &\!\!\! \!\int_s^r b(X_v^{\alpha,t,x}(\tilde \omega, .), \alpha_v(\tilde \omega, .))dv+(\int_s^r \sigma(X_v^{\alpha,t,x}, \alpha_v)dW_v) (\tilde \omega, .),
\end{eqnarray*}
on $[s,T]$  $P$-a.s. Now, by the first equality in Lemma \ref{zz}, 
there exists a $P$-null set ${\cal N}$  such that for each $ \omega \in {\cal N}^c$, setting   $\tilde \omega= $ $^{s}\omega$, we have
$$(\int_s^r \sigma(X_v^{\alpha,t,x}, \alpha_v)dW_v) (\tilde \omega, .) = 
\int_s^r \sigma(X_v^{\alpha(\tilde \omega,\cdot),s,\eta(\tilde \omega)}, \alpha_v(\tilde \omega,\cdot))dW^s_v\quad P-{\rm a.s.}\,,$$
which implies that the process $ (X_r^{\alpha,t,x}(\tilde \omega, \cdot))_{r \in [s,T]}$ is a solution
  of SDE \eqref{fofo}, and then, by  uniqueness of the solution of this SDE, we  have
  $X_r^{\alpha,t,x}(\tilde \omega, .)= X_r^{\alpha (\tilde\omega,.),s,\eta (\tilde\omega)},$ $s\leq r \leq T,
   $ $P$-a.s.\\ 
 Let us show the second assertion. First, note that since the filtration 
 ${\mathbb F}^s$ is the {\em completed} filtration of the natural filtration of 
$ W^s$ and $\Tilde{N}^s$ (with respect to the initial $\sigma$-algebra ${\cal B}(\Omega))$, we have a martingale representation theorem for 
  ${\mathbb F}^s$-martingales with respect to $ W^s$ and $\Tilde{N}^s$. Hence, there exists an unique solution
  $(Y_{r}^{ \alpha(\tilde \omega, \cdot),s,\eta(\tilde\omega)}, Z_{r}^{ \alpha(\tilde\omega, \cdot),s,\eta(\tilde\omega)}, K_{r}^{ \alpha(\tilde\omega, \cdot),s,\eta(\tilde\omega)})
 _{ s \leq r \leq T}$ in 
${\cal S}^2_s \times \mathbb{H}^2_s \times \mathbb{H}^2_{s, \nu}$ of the reflected BSDE on $\Omega \times[s,T]$ driven by $ W^s$ and $\Tilde{N}^s$ and associated with filtration $\mathbb{F}^s$ and with obstacle
 $\bar h(r, X_r^{\alpha(\tilde \omega,\cdot),s,\eta(\tilde \omega)})$.
  Equalities \eqref{abc} then
   follow from similar arguments as above together with the uniqueness of the solution of a Lipschitz RBSDE.
  Equality \eqref{abcd} is obtained by taking $r=s$ in equality \eqref{abc} and by using the definition of $u ^{\alpha (\tilde\omega,.)}$.
  \fproof

\begin{remark}\label{remarque} 
In the above proofs, we have treated the  {\em $P$-null sets} issues carefully. We stress that all the filtrations are {\em completed} with respect to 
${\cal B}(\Omega)$ and $P$. The underlying probability space is thus always  the {\em completion} of the initial probability space $(\Omega, {\cal B}(\Omega),P)$ (that is $(\Omega, {\cal F}_T,P)$). Note that the  {\em $P$-null sets} remain always the same, which is particularly important for stochastic integrals (see Lemma \ref{zz}), and also for BSDEs because the solution of a BSDE is unique up to a 
{\em $P$-null set}.\\
Moreover, in the proof  of Lemma \ref{zz}, the choice of the sequence of
step functions approximating the process $Z$  is appropriate to handle the issues of $P$-null sets.

Note that Theorem \ref{egal} applied to the simpler case when $\alpha$ $\in$ $\mathcal{A}_{s}^{s}$ ensures that the solution 
of \eqref{4.4} (with $t$ replaced by $s$) coincides on $[s, T] \times \Omega$ with the solution in ${\cal S}^2 \times \mathbb{H}^{2} \times \mathbb{H}^{2}_{\nu} $ of the reflected BSDE similar to \eqref{4.4} but 
driven by $W$ and $\tilde N$ instead of $W^s$ and $\tilde N^s$, and associated with  
${\mathbb F}$. 
\end{remark}

\subsection{Measurability properties and $\varepsilon$-optimal controls}\label{mesusection}

We need  to show a measurability property of the function $u^{ \alpha }(t,x)$ with respect  to 
control $\alpha$ and initial condition $x$. To this purpose, we first provide a preliminary result, which will allow us to handle the nonlinearity of the expectation.
\begin{proposition} \label{gborelian} 
Let $(\Omega, {\cal F}, P)$ be a probability space. 
For each $q\geq 0$, we denote by  $L^{q}$ the set 
$L^{q}(\Omega, {\cal F}, P) $. Suppose that the Hilbert space $L^2$ equipped with the usual scalar product is separable.

Let $g: \mathbb{R} \rightarrow  \mathbb{R}$ be a Borelian function such that $| g(x)| \leq C(1+ |x|^p)$ 
for each real $x$, with $p\geq 0$.
The map  $\varphi ^g$ defined by 
 \begin{equation*}
 \varphi ^g: L^{2p} \cap L^2 \rightarrow L^2; \,\,  \xi \mapsto g \circ \xi \,(\,= g(\xi))
 \end{equation*}
  is then ${\cal B}' ( L^{2p} \cap L^2) /  {\cal B}(L^2)$-measurable, 
where  ${\cal B}(L^2)$ is the Borelian $\sigma$-algebra on $L^2$, and ${\cal B}' ( L^{2p} \cap L^2 )$
  is the $\sigma$-algebra 
 induced by  ${\cal B}(L^2)$ on $L^{2p} \cap L^2$.
\end{proposition}
\noindent The proof of this proposition is postponed in the Appendix.\\
Using this result, we  now show the following measurability property.

\begin{theorem}\label{mesu} 
Let $s \in [0,T]$.
 The map $(\alpha,x) \mapsto u^{ \alpha } (s,x)$; ${\mathcal{A}}_s^s \times \mathbb{R} \rightarrow  \mathbb{R}$, 
 is  $\mathcal{B}'({\mathcal{A}}_s^s)
 \otimes
 \mathcal{B}(  \mathbb{R}) / \mathcal{B}(  \mathbb{R})$-measurable,
 where 
 $ \mathcal{B}'({\mathcal{A}}_s^s)$ denotes the $\sigma$-algebra 
 induced by  ${\cal B}({\mathbb H}_s^2)$ on ${\mathcal{A}}_s^s$.  
\end{theorem}

\dproof 
Recall that $ u^\alpha(s,x)={Y}_{s,T}^{\alpha,s,x}[g(X_{T}^{\alpha,s,x})]$ is also denoted by 
${Y}_{s,T}^{\alpha,s,x}[\bar h(.,X_{.}^{\alpha,s,x})]$.\\
Let $x^1, x^2$ $\in \mathbb{R}$, 
and $\alpha^1,\alpha^2 \in \mathcal{A}_s ^s.$
By classical estimates on diffusion processes and the assumptions made on the coefficients,  we get
\begin{equation}\label{xesti}
 \mathbb{E}[\sup_{r \geq s }|X_r^{\alpha^1,s,x^1}-X_r^{\alpha^2,s,x^2}|^2]\leq 
 C (\| \alpha^1-\alpha^2 \|^2_{\mathbb{H}_s^2} + |x^1- x^2|^2 ). 
 \end{equation}
  We introduce the map $\Phi: \mathcal{A}_s ^s \times \mathbb{R} \times 
  {\cal S}_s^2 \times L_s^2 \rightarrow {\cal S}_s^2$; 
 $(\alpha,x, \zeta_\cdot , \xi) \mapsto {Y}_{s,T}^{\alpha,s,x} [ \eta_\cdot , \xi]$, where   
 \\ ${Y}_{s,T}^{\alpha,s,x} [ \zeta_\cdot , \xi]$ denotes here the solution at time $s$
  of the  reflected BSDE associated with driver $f^{\alpha,s,x}:= (f(\alpha_r,r, X_{r}^{\alpha,s,x},.){\bf 1}
  _{r \geq s}
  ) $ 
  obstacle 
  $(\eta_s)_{s<T}$ and terminal condition $\xi$. \\
  By the estimates on RBSDEs (see the Appendix in \cite{DQS}), 
  using the Lipschitz property of 
$f$ w.r. to $x, \alpha$ and estimates \eqref{xesti}, 
for all $x^1, x^2$ $\in \mathbb{R}$, 
$\alpha^1,\alpha^2 \in \mathcal{A}_s ^s$, $\eta_\cdot^1,\eta_\cdot^2 \in S_s ^2$ and $\xi^1,\xi^2 \in L_s ^2$, we have
\begin{equation*}
 |Y_{s,T}^{\alpha^1,s,x^1}[\eta_\cdot^1, \xi^1]-Y_{s,T}^{\alpha^2,s,x^2}[\eta_\cdot^2, \xi^2]|^2\leq 
 C (\| \alpha^1-\alpha^2 \|^2_{\mathbb{H}_s^2} + |x^1- x^2|^2 +  \| \eta_\cdot^1-\eta_\cdot^2 \|^2_{S_s^2}+  \| \xi^1-\xi ^2 \|^2_{L_s^2}). 
 \end{equation*}
The map $\Phi$  is thus Lipschitz-continuous with respect to the norm 
  $\| \,. \,\|^2_{\mathbb{H}_s^2} + |\,.\,|^2  + \| \,. \,\|^2_{{\cal S}_s^2} + \| \,. \,\|^2_{L_s^2}$. 
  
 Recall that by assumption, $| h(t,x)| \leq C(1+ |x|^p)$, and that $h$
 is Lipschitz continuous with respect to $x$ uniformly in $t$.
 One can derive that  the map $S_s ^{2p} \cap S_s ^{2}  \rightarrow S_s^2$, 
 $\eta_\cdot \mapsto h(., \eta_\cdot)$ is Lipschitz-continuous for the norm $\| . \|^2_{S_s^2}$ and thus Borelian,   $S_s^2$ being equipped with the Borelian $\sigma$-algebra ${\cal B}(S_s^2)$ and 
  its sub-space $S_s ^{2p} \cap S_s ^{2}$ with the $\sigma$-algebra 
 induced by  ${\cal B}(S_s^2)$.

Moreover, by Lemma \ref{sepa}, the Hilbert space $L_s^2$  is separable. 
We can thus apply Proposition \ref{gborelian} and get that the map $L_s^{2p} \cap L_s^2 \rightarrow L_s^2$,  $\xi \mapsto g(\xi)$ is Borelian.
 

 We thus derive that 
 the map 
  $(\alpha,x) \mapsto (\alpha, x,h(.,X_{.}^{\alpha,s,x}), g (X_{T}^{\alpha,s,x}))$ defined on $\mathcal{A}_s ^s 
   \times \mathbb{R}   $ and 
  valued in 
  $\mathcal{A}_s ^s 
   \times \mathbb{R}  \times {\cal S}_s^2 \times L_s^2
 $ is $\mathcal{B}'({\mathcal{A}}_s^s) \otimes {\cal B}(\mathbb{R})/  \mathcal{B}'({\mathcal{A}}_s^s)    \otimes {\cal B}(\mathbb{R}) \otimes {\cal B}({\cal S}_s^2)  \otimes {\cal B}
(L_s^2) $-measurable. 
  By composition, it follows that the map 
$(\alpha,x) \mapsto  {Y}_{s,T}^{\alpha,s,x}[ h(.,X_{.}^{\alpha,s,x}), g (X_{T}^{\alpha,s,x})]$ $= u^\alpha(s,x)$ is measurable. 
\fproof

For each $(t,s)$ with $s \geq t$, we introduce the set $\mathcal{A}_s^t$ of restrictions to $[s,T]$ of the controls in $\mathcal{A}_t^t$. They can also be identified to the controls $\alpha$ in $\mathcal{A}_t^t$ which are equal to $0$ on $[t,s]$.


Let $\eta \in {L}^2(\mathcal{F}_{s}^t).$ 
Since $\eta$ is $\mathcal{F}_{s}$-measurable, up to a $P$-null set, it can be written as a 
measurable map, still denoted by $\eta$, of the past trajectory $^s\omega$ (see the argument used in the proof of Lemma \ref{sepa} for details).
For each $\omega \in {\Omega}$, by using the definition of the function $u$, we have:
\begin{equation}\label{aux}
u(s,\eta(^s \omega))=\sup_{\alpha \in \mathcal{A}_s^s} u^{\alpha}(s, \eta(^s\omega)).
\end{equation}
%

 
 By Theorem \ref{mesu} together with a measurable selection theorem, we show
 the existence of nearly optimal controls for \eqref{aux} satisfying some specific measurability properties.

\begin{theorem}\label{select}(Existence of $\varepsilon$-optimal controls)
Let $t \in [0,T]$, $s \in [t,T[$ and $\eta \in {L}^2(\mathcal{F}_s^t).$ Let $\varepsilon >0$. 
There exists 
$\alpha^{\varepsilon}$ $\in \mathcal{A}_s^t$ such that,
for almost every $ \omega \in \,  \Omega$,  
$\alpha^{\varepsilon}(^s \omega,T^s)$ is $\varepsilon$-optimal
 for Problem \eqref{aux},
 in the sense that  
 \begin{equation*}\label{optima}
u(s,\eta(^s \omega)) \,\leq \, u^{\alpha^{\varepsilon}(^s\omega,T^s)}(s, \eta(^s\omega))+\varepsilon.
\end{equation*}

\end{theorem}


\dproof  
 Without loss of generality, we may assume that $t=0$.  
We introduce the space 
$^s\Omega:=\{(\omega_r)_{0 \leq r \leq s}; \omega \in \Omega \}$, equipped with its Borelian $\sigma$-algebra denoted 
by ${\cal B}(^s\Omega)$,
and the probability measure $^sP$, which corresponds to the image of $P$  by $^s S: \Omega \rightarrow ^s\Omega;$ 
$\omega \mapsto (\omega_r)_{r \leq s}$.
The Hilbert space $\mathbb{H}_s^2$ of square-integrable predictable processes on $\Omega^s \times [s,T]$, equipped with the norm $\|\cdot\|_{\mathbb{H}_s^2}$ 
 is separable
 (see Lemma \ref{sepa}). 
Moreover, $\mathcal{A}_s^s$ is a closed subset of $\mathbb{H}^2_s$. Also, the space ${^s\Omega}$ of paths (RCLL) before $s$ is Polish for the Skorohod metric. 
Now, as seen above, since $\eta$ is $\mathcal{F}_s$-measurable, up to a $P$-null set, we can suppose that it is of the form
$\eta  \circ S^s$, where $\eta$ is ${\cal B}(^s\Omega)$-measurable.
Moreover, by Theorem \ref{mesu},  the map $(\tilde \omega, \alpha)$ $\mapsto$  
$u^{ \alpha}(s,\eta  (\tilde \omega))$ is  $ \mathcal{B} 
({^s}\Omega) \otimes \mathcal{B}(\mathcal{A}_s^s )$-measurable with respect to $(x, \alpha)$.
We can thus apply Proposition 7.50 in \cite{BS} 
 to the problem \eqref{aux}. Hence,
there exists a map  ${\underline \alpha}^{\varepsilon}:$ ${^s}\Omega \mapsto \mathcal{A}_s^s$\,; $\,\tilde \omega\mapsto {\underline \alpha}^{\varepsilon}(\tilde \omega, \cdot)$, which is  {\em universally} measurable, that is $ \mathcal{U}
({^s}\Omega) / \mathcal{B}(\mathcal{A}_s^s)$-measurable, and such that
\begin{equation*}
u(s,\eta (\tilde \omega)) \leq u^{ {\underline \alpha}^{\varepsilon}(\tilde \omega,\cdot)}(s, \eta (\tilde \omega))+\varepsilon  
\quad {\rm for\,\, all}\,\, \tilde \omega \in {^s}\Omega.
\end{equation*}  
Here, $\mathcal{U}
({^s}\Omega)$ denotes the {\em universal} $\sigma$-algebra on ${^s}\Omega$.
Let us now apply Lemma \ref{Cra} to $X= {^s}\Omega$, to 
 $E=\mathbb{H}_s^2$ and to probability $Q= P^s$. By definition of $\mathcal{U}
({^s}\Omega)$ (see e.g. \cite{BS}), we have $ \mathcal{U}({^s}\Omega) \subset \mathcal{B}_{
 Q} ({^s}\Omega),$ where $\mathcal{B}_{Q} ({^s}\Omega)$ denotes {\em the completion} 
of $\mathcal{B} ({^s}\Omega)$ with respect to $Q$. Hence, 
there exists a map \\
 ${\hat \alpha}^{\varepsilon}:$ 
${^s}\Omega \mapsto \mathcal{A}_s^s$\,; $\,\tilde \omega\mapsto {\hat \alpha}^{\varepsilon}(\tilde \omega, \cdot)$ which is  Borelian, that is $ \mathcal{B}
({^s}\Omega) / \mathcal{B}(\mathcal{A}_s^s)$-measurable, and such that
\begin{equation*}
{\hat \alpha}^{\varepsilon}(\tilde \omega,\cdot)= {\underline \alpha}^{\varepsilon}(\tilde \omega,\cdot)
 \quad {\rm for} \quad { ^s P}-{\rm almost\,\, every}\,\, \tilde \omega \in {^s}\Omega.
\end{equation*}  
Since ${\mathbb{H}_s^2}$ is a separable Hilbert space, for each $\tilde \omega$, we have 
${\hat \alpha}^{\varepsilon}_u(\tilde \omega, \omega)=\sum_{i}\beta^{i, \varepsilon}(\tilde \omega)
e^i_u(\omega)$ $dP(\omega) \otimes du$-a.s. , where
 $\beta^{i,\varepsilon}(\tilde \omega)=<{\hat \alpha}^{\varepsilon}(\tilde \omega,\cdot),e^{i}(\cdot)>_{\mathbb{H}_s^2}$ and $\{e^{i}, i \in \mathbb{N} \}$ is a countable orthonormal basis of $\mathbb{H}_s^2$. Note that $\beta^{i,\varepsilon}$ is Borelian, that is $ \mathcal{B}({^s}\Omega)/
  \mathcal{B}({\mathbb R})$-measurable.\\
  Let $\bar { \alpha}^{\varepsilon}:$ $^s \Omega \mapsto \mathcal{A}_s^s$\,; $\,\tilde \omega\mapsto \bar{ \alpha}^{\varepsilon}(\tilde \omega, \cdot)= \sum_{i}\beta^{i, \varepsilon}(\tilde \omega)e^i(\cdot)$. It is Borelian, that is $ \mathcal{B}
({^s}\Omega) / \mathcal{B}(\mathcal{A}_s^s)$-measurable.\\
We now define a process $\alpha^{\varepsilon}$ on $[0,T] \times \Omega$
by $ \alpha^{\varepsilon}_r(\omega):= \sum_{i}\beta^{i, \varepsilon}( S^s (\omega))e^i(\omega)$.
It remains to prove that it is  $\mathcal{P}$-measurable.
Note that $\beta^{i,\varepsilon}\circ S^s$ is $\mathcal{F}_s$-measurable by composition. 
Since the process $(e_u^{i})_{s \leq u \leq T}$  is $\mathcal{P}^{s}$-measurable, 
the process 
$(\beta^{i,\varepsilon}\circ S^s )\, e_u^{i}$ is $\mathcal{P}$-measurable. Indeed, if we take $e^i$ of the form $e^i_u=H \textbf{1}_{]r,T]}(u)$ with $r \geq s$ and $H$ a random variable $\mathcal{F}_r^s$-measurable, then the random variable $(\beta^{i,\varepsilon}\circ $$S ^s) \, H$ 
is $\mathcal{F}_r$-measurable and hence the process $(\beta^{i,\varepsilon}\circ $
$S ^s) \, H \textbf{1}_{]r,T]}$ is $\mathcal{P}$-measurable. The process
 $\alpha^{\varepsilon}$ is thus $\mathcal{P}$-measurable.\\ 
Note also that $\alpha^{\varepsilon}(\tilde \omega,T^s(\omega))= \sum_{i}\beta^{i, \varepsilon}(\tilde \omega)e^i(\tilde \omega, \omega)$. Now, we have  $e^i(\tilde \omega, T^s(\omega))= e^i (\omega)$ 
because $e^i (\omega)$ depends on $ \omega$ only through $T^s(\omega)$.
Hence, $\alpha^{\varepsilon}(\tilde \omega,T^s(\omega))= \bar{ \alpha}^{\varepsilon}(\tilde \omega, 
\omega)$, which completes the proof.
\fproof

\subsection{A Fatou lemma for reflected BSDEs}\label{FatouR}

We establish  a Fatou lemma  for reflected  BSDEs, where the limit involves both terminal condition and terminal time.
This result will be used  to prove a {\em super (resp. sub)--optimality principle} involving the l.s.c. (resp. u.s.c.) envelope of 
the value function $u$ (see Theorem \ref{IMP}). We first introduce some notation.

A function $f$ is said to be a {\em Lipschitz driver} if \\
$f: [0,T]  \times \Omega \times \R^2 \times L^2_\nu \rightarrow \R $
$(\omega, t,y, z, k(\cdot)) \mapsto  f(\omega, t,y, z, k(\cdot))  $
  is $ {\cal P} \otimes {\cal B}(\R^2)  \otimes {\cal B}(L^2_\nu) 
- $ measurable,  uniformly Lipschitz with respect to $y, z, k(\cdot)$ and such that $f(.,0,0,0) \in \H^2$.

A Lipschitz driver $f$ is said to satisfy Assumption~\ref{Royer} if the following holds:
\begin{assumption}\label{Royer} 
Assume that  $dP \otimes dt$-a.s\, for each $(y,z, k_1,k_2)$ $\in$ $ \mathbb{R}^2 \times (L^2_{\nu})^2$,
$$f( t,y,z, k_1)- f(t,y,z, k_2) \geq \langle \gamma_t^{y,z, k_1,k_2}  \,,\,k_1 - k_2 \rangle_\nu,$$ 
with
$
\gamma:  [0,T]  \times \Omega\times \mathbb{R}^2 \times  (L^2_{\nu})^2  \rightarrow  L^2_{\nu}\,; \, (\omega, t, y,z, k_1, k_2) \mapsto 
\gamma_t^{y,z,k_1,k_2}(\omega,.)
$, supposed to be
 ${\cal P } \otimes {\cal B}({\mathbb R}^2) \otimes  {\cal B}( (L^2_{\nu})^2 )$-measurable, uniformly bounded in $L^2_{\nu}$, and satisfying $ dP(\omega)\otimes dt \otimes d\nu(e)$-a.s.\,, for each $(y,z, k_1, k_2)$ $\in$ ${\mathbb R}^2 \times (L^2_{\nu})^2$, the inequality $\gamma_t^{y,z, k_1, k_2} (\omega, e)\geq -1$.
\end{assumption}

This assumption ensures  the comparison theorem for BSDEs  with jumps  (see \cite{16} Th 4.2).

Let $(  \eta_t)$ be a given  RCLL obstacle process in $\mathcal{S}^2$ and let $f$ be a given Lipschitz driver.
 In the following, we will 
consider the case when the terminal time 
  is  a stopping time $ \theta$ $ \in {\cal T}$ and the terminal condition is a random variable  $\xi$ in $ L^2({\cal F}_ \theta)$. 
In this case,  the solution, denoted $(Y_{., \theta}(\xi), Z_{., \theta}(\xi), k_{., \theta}(\xi))$, of the {\em reflected BSDEs associated with terminal stopping time} $\theta$, driver $f$, obstacle $(  \eta_s)_{ s < \theta}$, and terminal condition $\xi$  
 is defined  as the unique solution in $\mathcal{S}^2 \times \mathbb{H}^2 \times
\mathbb{H}^2_{\nu}$ of the reflected BSDE with  terminal time $T$, driver $f(t,y, z,k) {\bf 1}_ { \{ t \leq  \theta  \}  }$, 
 terminal condition $\xi$ and obstacle $ \eta_t {\bf 1}_{t < \theta} + \xi {\bf 1}_{t \geq \theta}$.
 Note that 
$Y_{t, \theta}(\xi) = \xi, Z_{t, \theta}(\xi) =0,k_{t, \theta}(\xi)=0 $ for $t \geq  \theta$.

We first prove a continuity property for reflected BSDEs where the limit involves both terminal condition and terminal time.

  \begin{proposition}[A continuity property for reflected BSDEs] \label{conty}
Let $T>0$.  Let $(  \eta_t)$ be an RCLL process in $\mathcal{S}^2.$ Let $f$ be a given Lipschitz driver. 
Let $(\theta^n)_{ n \in \mathbb{N}} $ be a non increasing sequence of 
stopping times in 
$\mathcal{T}$,  converging a.s. to  $\theta \in \mathcal{T}$ as 
$n$ tends to $\infty$. Let   $(\xi^{n})_{ n \in \mathbb{N}} $ be a sequence of random variables such that 
$ \mathbb{E}[\sup_{n}( \xi^n)^2]< + \infty$, and  for each $n$,
$\xi^{n}$ is ${\mathcal F}_{\theta^{n}}$-measurable. Suppose  that $\xi^{n}$ converges a.s. to an ${\mathcal F}_{\theta}$-measurable random variable $\xi$ as 
$n$ tends to $\infty$. Suppose that 
\begin{equation}\label{ass}
  \eta_{\theta} \leq \xi \quad {\rm a.s.}
\end{equation}
Let $Y_{.,\theta^{n}}(\xi^n )$; $Y_{.,\theta}(\xi )$ be the solutions of the reflected BSDEs associated with driver $f$, obstacle $(  \eta_s)_{ s < \theta^n}$ (resp. $(  \eta_s)_{ s < \theta}$) , terminal time $\theta^n$ (resp. $\theta$), terminal condition $\xi^n$ (resp. $\xi$). We have
 $$Y_{0,\theta}(\xi ) = \lim_{n \rightarrow + \infty} Y_{0,\theta^{n}}(\xi^n ) \quad a.s.$$
When for each $n$,  $\theta_n= \theta$ a.s.\,, the result still holds without Assumption \eqref{ass}.

 \end{proposition}

By similar  arguments as in the Brownian case (see e.g.  \cite{10}), one can prove the following estimate on reflected BSDEs, which will be used in the proof of the above Proposition.
  \begin{lemma}\label{oubli1}
Let $\xi^{1}, \xi^{2} \in L^2({\cal F}_T)$ and $(\eta_t^{1})$, $(\eta_t^{2})\in \mathcal{S}^2$. 
Let $f^1, f^2$ be Lipschitz drivers with Lipschitz constant 
$C>0$. For $i=1,2$, let $(Y^i,Z^{i},k^{i}, A^{i})$ be the solution of the reflected BSDE 
with driver $f^i$, 
terminal time $T$, obstacle $(\eta_t^i)$ and terminal condition $\xi^i$.  
For $s \in [0,T]$, let 
$\overline{Y}_s:=Y_s^1-Y_s^2$, $\overline{\eta}_s:=\eta_s^1-\eta_s^2$, $\overline{\xi}:=
\xi^1-\xi^2$ and
 $\overline{f}(s):=f^1(s,Y_s^2,Z_s^2,k_s^2)-f(s,Y_s^2,Z_s^2,k_s^2)$.  Then, we have
%
\begin{equation}\label{eqA.2}
\| \overline{Y}\|_{{\cal S}^2}^2  \leq K \left(
 {\mathbb E}[\overline{\xi}^2]+{\mathbb E}[\int_0^T\overline{f}^2(s)ds]\right)+ \phi  \,\, \| \sup_{0\leq s < T}
 |\overline{\eta}_s| \|_{L^2},
\end{equation}
where the constant $K$ is universal, that is depends only on the Lipschitz constant $C$ and
$T$, and where the constant 
$\phi$ depends only on $C,T$, $\| \eta^i \|_{{\cal S}^2}$, $\| \xi^i \|_{L^2}$ and 
$\| {f}^i(s,0,0,0)\|_{\H^2}$, $i=1,2$.
\end{lemma}
%
\noindent {\bf Proof of Proposition \ref{conty}}. Let $n \in {\mathbb N}$.
 We apply \eqref{eqA.2} with 
 $f^1= f {\bf 1}_{t \leq \theta^n}$, $f^2= f {\bf 1}_{t \leq \theta}$, $\xi^1 = \xi^n$, $\xi^2 = \xi$, $\eta_t^1 = \eta_t {\bf 1}_{t < \theta ^n}
 + \xi^n {\bf 1}_{\theta ^n \leq t < T}$ and 
 $\eta^2_t = \eta_t {\bf 1}_{t < \theta}+ \eta_{\theta} {\bf 1}_{ \theta \leq t < \theta ^n } + \xi {\bf 1}_{\theta ^n \leq t < T}$. 
 We have 
  $Y_\cdot^1= Y_{.,\theta^n}(\xi^n )$ a.s. 
 Moreover,  since by assumption $\eta_{\theta} \leq \xi$ a.s.\,,
we have $Y^2_\cdot = Y_{.,\theta}(\xi )$ a.s. 
Note that $(Y^2_t, Z^2_t, k^2_t) = (\xi,0,0)$ a.s. on $\{t \geq \theta\}$.
 We thus obtain
  \begin{equation}\label{secundo}
| Y_{0,\theta^n}(\xi^n )-  Y_{0,\theta}(\xi ) |^2\leq K \left(
 {\mathbb E}[  (\xi^n  - \xi) ^2]+{\mathbb E}[\int_{\theta}^{\theta^n}{f}^2(s, \xi, 0,0)ds] \right)+ \phi  \,\, \| \sup_{\theta \leq s
  < \theta^n}
 |{\eta}_s- \eta_{\theta}| \|_{L^2},
\end{equation}
where the constant $K$ depends only on the Lipschitz constant $C$ of $f$ and
the terminal time $T$, and where the constant 
$\phi$ depends only on $C$, $T$, $\| \eta \|_{{\cal S}^2}$, $\sup_n \| \xi^n \|_{L^2}$  and 
$\| {f}(s,0,0,0)\|_{\H^2}$. 
Since the obstacle $ (\eta_t)$ is right-continuous and $\theta^n \downarrow \theta$ a.s.\,, we have
$\lim_{n \rightarrow + \infty} \| \sup_{\theta \leq s
  \leq  \theta^n}
 |{\eta}_s- \eta_{\theta}| \|_{L^2} =0.$ 
 The right member of  \eqref{secundo} thus tends to $0$ as $n$ tends to $+ \infty$.
The result follows.
\fproof

\begin{remark}
 Compared with the case of non reflected BSDEs (see Proposition A.6 in \cite{16}),  there is an extra difficulty due to the presence 
 of the obstacle (and the variation of the terminal time). The additional assumption  \eqref{ass} on the obstacle is here required to obtain the result.
\end{remark}

 \noindent Using Proposition \ref{conty}, we derive a Fatou lemma in the reflected case, where the limit involves both terminal condition and terminal time. 

\begin{proposition}[A Fatou lemma for reflected BSDEs]\label{fatou2}
Let $T>0$.  Let $(  \eta_t)$ be an RCLL process in $\mathcal{S}^2.$ Let $f$ be a  Lipschitz driver satisfying Assumption \ref{Royer}. 
Let $(\theta^n)_{ n \in \mathbb{N}} $ be a non increasing sequence of stopping times in 
$\mathcal{T}$,  converging a.s. to  $\theta \in \mathcal{T}$ as 
$n$ tends to $\infty$. Let  $(\xi^{n})_{ n \in \mathbb{N}} $ be a sequence of random variables such that 
$ \mathbb{E}[\sup_{n}( \xi^n)^2]< + \infty$, and  for each $n$,
$\xi^{n}$ is ${\mathcal F}_{\theta^{n}}$-measurable.\\
Let $Y_{.,\theta^{n}}(\xi^n )$ ; $Y_{.,\theta}(\liminf_{n \rightarrow + \infty} \xi^{n} )$  and $Y_{.,\theta}(\limsup_{n \rightarrow + \infty} \xi^{n} )$ be the solution(s) of the reflected BSDE(s) associated with driver $f$, obstacle $(  \eta_s)_{ s < \theta^n}$ (resp. $(  \eta_s)_{ s < \theta}$) , terminal time $\theta^n$ (resp. $\theta$), terminal condition $\xi^n$ (resp.  $\liminf_{n \rightarrow + \infty} \xi^{n}$ and $\limsup_{n \rightarrow + \infty} \xi^{n}$).\\
Suppose that
\begin{equation}\label{ij}
\liminf_{n \rightarrow + \infty} \xi^{n} \geq   \eta_{\theta} \quad ({\rm resp.}\quad  \limsup_{n \rightarrow + \infty} \xi^{n} \geq   \eta_{\theta}) 
\quad {\rm a.s.}
\end{equation}   
 $${\rm then}\quad \quad Y_{0,\theta}(\liminf_{n \rightarrow + \infty} \xi^{n} ) \leq \liminf_{n \rightarrow + \infty} Y_{0,\theta^{n}}(\xi^n ) \quad \left({\rm resp.}\quad Y_{0,\theta}(\limsup_{n \rightarrow + \infty} \xi^{n} ) \geq \limsup_{n \rightarrow + \infty} 
 Y_{0,\theta^{n}}(\xi^n ) \right).
$$
When for each $n$,  $\theta_n= \theta$ a.s.\,, the result still holds without Assumption \eqref{ij}.
  \end{proposition}
 
  
  \dproof
We present only the proof of the first inequality, since the second one is obtained by similar arguments.
For all $n$, we have by the monotonicity of reflected BSDEs with respect to terminal condition,
$Y_{0,\theta^{n}} ( \inf_{p \geq n}  \xi^{p}) \leq Y_{0,\theta^{n}} (\xi^{n}).$
We derive that
 $$  \liminf_{n \rightarrow + \infty}
 Y_{0,\theta^{n}} (\xi^{n}) \geq \liminf_{n \rightarrow + \infty}Y_{0,\theta^{n}} ( \inf_{p \geq n}   \xi^{p}) =
Y_{0,\theta}(\liminf_{n \rightarrow + \infty} \xi^{n} ),$$
where the last equality follows from Assumption \eqref{ij} together with Proposition \ref{conty}.
\fproof



\subsection{A {\em weak} dynamic programming principle}



We  will now provide a {\em weak} dynamic programming principle, that is both a  (weak) {\em sub- and  super-optimality principle of dynamic programming}, involving respectively  the 
upper semicontinuous envelope  $u^*$ and the lower semicontinuous envelope  $u_*$ of the value function $u$, defined 
by
$$
u^*(t,x) := \limsup_{(t',x') \rightarrow (t,x)} u(t',x'); \quad 
 u_*(t,x) := \liminf_{(t',x') \rightarrow (t,x)} u(t',x')  \quad \forall (t, x) \in [0,T] \times {\mathbb R}.$$
 We now define the maps $\bar u^*$ and $\bar u_*$ for each $(t, x) \in [0,T] \times {\mathbb R}$ by
 $$
\bar u^*(t,x) :=  u^*(t,x) {\bf 1}_{t < T} + g(x) {\bf 1}_{t =  T}; \quad
\bar u_*(t,x) := u_*(t,x) {\bf 1}_{t < T} + g(x) {\bf 1}_{t =  T}. 
$$
 Note that the functions $\bar u^*$ and $\bar u_*$ are Borelian. We have  $\bar u_* \leq u \leq \bar u^*$ and $\bar u_* (T,.)= u (T,.)=  \bar u^*(T,.)= g(.)$. Note that $ \bar u^*$ (resp. $\bar u_*$)  is not necessarily 
  upper (resp. lower) semicontinuous on $ [0,T] \times {\mathbb R}$, since the  terminal reward  $g$ is only Borelian. 
 

To prove the {\em weak} dynamic programming principle, we will use the splitting properties  (Th. \ref{egal}), the existence of   $\varepsilon$-optimal controls (Th. \ref{select}) and the Fatou lemma for RBSDEs (Prop. \ref{fatou2}).
\begin{theorem}[A {\em weak} dynamic programming principle]\label{IMP}
The value function $u$  satisfies the following weak {\em sub--optimality principle of dynamic programming}: \\
for each $t \in [0,T]$ and for each stopping time $\theta \in \mathcal{T}^t_{t},$ we have
%
%
\begin{equation}\label{DPP}
u(t,x) \leq {\sup_{\alpha \in \mathcal{A}_t^t}\sup_{\tau \in \mathcal{T}_{t}^t}}
\mathcal{E}_{t, \theta \wedge \tau }^{\alpha,t,x}\left[h(\tau, X_{\tau}^{\alpha,t,x})\textbf{1}_{\tau <\theta}+\bar u^*(\theta,X_{\theta}^{\alpha,t,x})\textbf{1}_{\tau \geq\theta}\right], 
\end{equation}
Moreover, the  following weak {\em super--optimality principle of dynamic programming} holds: \\
for each $t \in [0,T]$ and for each stopping time $\theta \in \mathcal{T}^t_{t},$ we have 
\begin{equation}\label{DPP2}
u(t,x) \geq {\sup_{\alpha \in \mathcal{A}_t^t}\sup_{\tau \in \mathcal{T}_{t}^t}}
\mathcal{E}_{t, \theta \wedge \tau }^{\alpha,t,x}\left[h(\tau, X_{\tau}^{\alpha,t,x})\textbf{1}_{\tau <\theta}+ \bar u_*(\theta,X_{\theta}^{\alpha,t,x})\textbf{1}_{\tau \geq\theta}\right]. 
\end{equation}
\end{theorem}
 \begin{remark}\label{R}
The proof given below also shows that this {\em weak} DPP still holds with $\theta$ replaced by $\theta^{\alpha}$ in inequalities \eqref{DPP} and \eqref{DPP2}, given a family of stopping times indexed by controls $\{ \theta^{\alpha}, \alpha \in \mathcal{A}_t^t\}.$

Note that no regularity condition is required on  $g$ to ensure this {\em weak} DPP, even \eqref{DPP2}. This is not the case in the  literature even for classical expectation (see \cite{BT}, \cite{BN},\cite{BY}). 
Moreover, our DPPs are stronger 
 than those given in these papers, where inequality \eqref{DPP} (resp. \eqref{DPP2}) is established with $u^*$ (resp. $u_*$) instead of $\bar u^*$ (resp. $\bar u_*$). Now, $\bar u^* \leq u^*$ and $\bar u_* \geq u_*$.

\end{remark}

Before giving the proof, we introduce the following notation. For each $ \theta \in {\cal T}$ and each  $\xi$ in $ L^2({\cal F}_\theta)$,
we denote by $(Y^{\alpha,t,x}_{.,  \theta}(\xi), Z^{\alpha,t,x}_{.,  \theta}(\xi), k^{\alpha,t,x}_{.,  \theta}(\xi))$ 
 the unique solution in $\mathcal{S}^2 \times \mathbb{H}^2 \times
\mathbb{H}^2_{\nu}$ of the reflected BSDE with  driver $f^{\alpha,t,x}{\bf 1}_ { \{ s \leq  \theta  \}  }$,
 terminal time $T$, 
 terminal condition $\xi$ and obstacle 
 $h(r, X_r^{\alpha,t,x}) {\bf 1}_{r < \theta} + \xi {\bf 1}_{r \geq \theta}$.
 
\dproof 
By estimates for reflected BSDEs (see Prop. 5.1 in \cite{DQS}),
 the function $u$ has at most polynomial growth at infinity.
Hence, the random variables $ \bar u ^*(\theta,X_{\theta}^{\alpha,t,x})$ and $\bar u _*(\theta,X_{\theta}^{\alpha,t,x})$ are square 
integrable.
Without loss of generality,  to simplify notation, we suppose that $t=0$. \\
We first show the second assertion (which is the most difficult), or equivalently:
\begin{equation}\label{sense2}
 {\sup_{\alpha \in \mathcal{A}}} \,Y_{0,\theta}^{\alpha,0,x}\left[\bar u _*(\theta,
X_{\theta}^{\alpha,0,x})\right]  \leq u(0,x), \quad \forall \theta \in \mathcal{T}.
\end{equation}
 Let $\theta \in \mathcal{T}.$
For each $n \in \mathbb N$, we define 
\begin{equation}\label{thetan}
\theta^n:=\sum_{k=0}^{2^n-1}t_k\textbf{1}_{A_k}+T \textbf{1}_{\theta=T},
\end{equation}
where  $t_k:=\frac{(k+1)T}{2^n}$ and $A_k:= \{\frac{kT}{2^n} \leq \theta <\frac{(k+1)T}{2^n}\}$. Note that 
$\theta^n \in \mathcal{T}$ and $\theta^n \downarrow \theta$.\\
On $\{ \theta=T \}$ we have $\theta^n=T$ for each $n$. We thus get 
$\bar u _*(\theta^n, X_{\theta^n}^{\alpha,0,x})=  \bar u _*(\theta, X_{\theta}^{\alpha,0,x})$ for each $n$ on $\{ \theta=T \}$.
Moreover, on $\{ \theta<T \}$, the lower semicontinuity of $\bar u _*$ on $[0,T[ \times \R$ together with the right continuity of the process $X^{\alpha,0,x}$ implies that 
$$\bar u _*(\theta,X_{\theta}^{\alpha,0,x}) \leq \liminf_{n \rightarrow + \infty } \bar u _*(\theta^n,X_{\theta^n}^{\alpha,0,x}) \quad {\rm a.s.}\, $$
Hence, by the comparison theorem for reflected BSDEs, we get:
\begin{equation*} 
Y_{0, \theta}^{\alpha,0,x}\left[\bar u _*(\theta,X_{\theta}^{\alpha,0,x})\right]  \leq Y_{0, \theta}^{\alpha,0,x}\left[ \liminf_{n \rightarrow + \infty } \bar u _*(\theta^n,X_{\theta^n}^{\alpha,0,x}) \right].
\end{equation*}
On $\{ \theta<T \}$, we have 
 $$\lim \inf_{n \rightarrow \infty} \bar u _*(\theta^n, X_{\theta^n}^{\alpha,0,x}) \geq \lim \inf_{n \rightarrow \infty} \overline{h}(\theta^n, X_{\theta^n}^{\alpha,0,x}) = \lim_{n \rightarrow \infty}h(\theta^n, X_{\theta^n}^{\alpha,0,x}) =h(\theta, X_{\theta}^{\alpha,0,x}) \quad \text{ a.s. }$$
by the regularity properties of $h$ on $[0,T[ \times \mathbb{R}.$
On $\{ \theta=T \}$,   $\theta^n=T$ and 
$$\bar u _*(\theta^n, X_{\theta^n}^{\alpha,0,x})= \bar u _*(T, X_{T}^{\alpha,0,x})=g(X_{T}^{\alpha,0,x})=\bar{h}(T, X_{T}^{\alpha,0,x}).  $$
Hence, we have
$\liminf_{n \rightarrow +\infty} \bar u _*(\theta^n, X_{\theta^n}^{\alpha,0,x}) \geq \bar{h}(\theta, X_\theta^{\alpha,0,x}) \text{ a.s.}$ Condition \eqref{ij} is thus satisfied with $\xi^n = \bar u _*(\theta^n, X_{\theta^n}^{\alpha,0,x})$ and $\xi _t= \bar{h}(t, X_t^{\alpha,0,x})$.
We can thus apply the  Fatou lemma for  reflected BSDEs 
(Prop. \ref{fatou2}). We thus get: 
\begin{equation} \label{eq1bis} 
Y_{0, \theta}^{\alpha,0,x}\left[\bar u _*(\theta,X_{\theta}^{\alpha,0,x})\right]  \leq Y_{0, \theta}^{\alpha,0,x}\left[ \liminf_{n \rightarrow + \infty } \bar u _*(\theta^n,X_{\theta^n}^{\alpha,0,x}) \right]
\leq \liminf_{n \rightarrow \infty}Y_{0, \theta^n}^{\alpha,0,x}\left[\bar u _* (\theta^n,X_{\theta^n}^{\alpha,0,x})\right].
\end{equation}
Let $\varepsilon >0$. Fix $n \in \mathbb N$. 
For each $k< 2^n -1$, let $\mathcal{A}_{t_k}$ be the set of the restrictions to $[t_k, T]$ of the controls $\alpha$ in $\mathcal{A}$. By Theorem \ref{select}, there exists a $P$-null set ${\cal N}$ (which depends on $n$ and $\varepsilon$) such that for each $k< 2^n -1$, there exists an $\varepsilon$-optimal control 
control $\alpha^{n,\varepsilon, k}$   in $\mathcal{A}_{t_k}^0$ $(= \mathcal{A}_{t_k})$  for the control problem at time $t_k$ 
with initial condition $\eta=  X_{t_k}^{\alpha,0,x}$, that is
satisfying the inequality 
\begin{equation}\label{optimab}
 u (t_k, X_{t_k}^{\alpha,0,x}(^{t_k}\omega)) \,\leq \, u^{\alpha^{n,\varepsilon, k}(^{t_k}\omega,\cdot)}(t_k, X_{t_k}^{\alpha,0,x}
(^{t_k}\omega))+\varepsilon 
\end{equation}
for each $ \omega \in \,  {\cal N}^c$.
Using the definition of the maps $u^{\alpha^{n,\varepsilon, k}(^{t_k}\omega,\cdot)}$ 
together with the splitting property for reflected BSDEs \eqref{abcd}, we derive that 
 there exists a $P$-null set ${\cal N}$ which contains the above one such that for each $ \omega \in \,  {\cal N}^c$ and for each $k< 2^n -1$, we have
\begin{align*}
u ^{\alpha^{n,\varepsilon,k}(^{t_k}\omega,\cdot)}(t_k, X_{t_k}^{\alpha,0,x}(^{t_k}\omega))
=\,Y_{t_k,T}^{\alpha^{n,\varepsilon,k}(^{t_k}\omega,\cdot),t_k,X_{t_k}^{\alpha,0,x}(^{t_k}\omega)}
 =\,Y_{t_k,T}^{\alpha^{n,\varepsilon,k},t_k,X_{t_k}^{\alpha,0,x}}
(^{t_k}\omega). 
\end{align*}
Here, $Y_{.,T}^{\alpha^{n,\varepsilon,k},t_k,X_{t_k}^{\alpha,0,x}}= Y_{.,T}^{f^{\alpha^{n,\varepsilon,k},t_k,X_{t_k}^{\alpha,0,x}}}
[\bar h(r, X_r^{{\alpha}^{n,\varepsilon,k}, t_k, X_{t_k}^{\alpha,0,x}})]
$ denotes the solution 
of the reflected BSDE associated with terminal time $T$, obstacle $ ( \bar h(r, X_r^{{\alpha}^{n,\varepsilon,k}, t_k, X_{t_k}^{\alpha,0,x}}))_{t_k \leq r \leq T}$ and driver\\ 
$f^{\alpha^{n,\varepsilon,k},t_k,X_{t_k}^{\alpha,0,x}}(r, y,z,k):= f(\alpha^{n,\varepsilon,k}_r,r, X_r^{\alpha,t_k,X_{t_k}^{\alpha,0,x}},y,z,k)$. \\

Set
$\alpha_{s}^{n,\varepsilon}:=\sum_{k< 2^n-1} \alpha^{n,\varepsilon,k}_s\textbf{1}_{A_k}  + \alpha_s \textbf{1}_{\{\theta_n=T \}}$. 
Since for each $k$, $A_k$ $\in$ $\mathcal{F}_{t_k}$, there exists a $P$-null set ${\cal N}$ such that, on ${\cal N}^c$, for each $k< 2^n -1$, 
we have the following equalities: 
\begin{align*}
Y_{t_k,T}^{\alpha^{n,\varepsilon,k},t_k,X_{t_k}^{\alpha,0,x}} \textbf{1}_{A_k}  
&=\,Y_{t_k,T}^{f^{\alpha^{n,\varepsilon,k},t_k,X_{t_k}^{\alpha,0,x}}\textbf{1}_{A_k} }[\bar h(r, X_r^{{\alpha}^{n,\varepsilon,k}, t_k, X_{t_k}^{\alpha,0,x}})\textbf{1}_{A_k} ]  \nonumber \\
&=\,Y_{t_k,T}^{f^{\alpha^{n,\varepsilon},\theta^n,X_{\theta^n}^{\alpha,0,x}}\textbf{1}_{A_k} }[\bar h(r, X_r^{{\alpha}^{n,\varepsilon}, \theta^n, X_{\theta^n}^{\alpha,0,x}}) \textbf{1}_{A_k} ],
\end{align*}
where, for a given driver $f$, $Y^{f \textbf{1}_{A_k}}$ denotes  the solution of the reflected BSDE 
associated with $f\textbf{1}_{A_k}$. 
We thus get $Y_{t_k,T}^{\alpha^{n,\varepsilon,k},t_k,X_{t_k}^{\alpha,0,x}} \textbf{1}_{A_k}=\, Y_{\theta^n,T}^{\alpha^{n,\varepsilon},\theta^n,X_{\theta^n}^{\alpha,0,x}} \textbf{1}_{A_k}$ on ${\cal N}^c$.
Using inequalities \eqref{optimab}, we get 
\begin{equation*}\label{opo}
\bar u _* (\theta^n, X_{\theta^n}^{\alpha,0,x})=\sum_{0 \leq k< 2^n-1} u _*(t_k, X_{t_k}^{\alpha,0,x})\textbf{1}_{A_k}+  g ( X_T^{\alpha,t,x}) 
 \textbf{1}_{\{\theta_n=T \}}
\leq Y_{\theta^n,T}^{\alpha^{n,\varepsilon},\theta^n,X_{\theta^n}^{\alpha,0,x}}  + \varepsilon  \  \text{ on } \ {\cal N}^c. 
\end{equation*} 
%
We set:
$
\Tilde{\alpha}^{n,\varepsilon}_s:=\alpha_s \textbf{1}_{s<\theta^n}+\alpha_{s}^{n,\varepsilon}\textbf{1}_{\theta^n \leq s \leq T}.$
Note that $\Tilde{\alpha}^{n,\varepsilon} \in \mathcal{A}.$
Using the comparison theorem together with the estimates on reflected BSDEs (see \cite{DQS}), we obtain
$$Y_{0, \theta^n}^{\alpha,0,x}[ \bar u _* (\theta^n,X_{\theta^n}^{\alpha,0,x}) ] \leq 
Y_{0, \theta^n}^{\alpha,0,x}[Y_{\theta^n,T}^{\alpha^{n,\varepsilon},\theta^n,X_{\theta^n}^{\alpha,0,x}}] + K\varepsilon =  Y_{0,T}^{\Tilde{\alpha}^{n,\varepsilon },0,x}+ K\varepsilon ,$$ 
where the last equality follows from the flow property. Since
$Y_{0,T}^{\Tilde{\alpha}^{n,\varepsilon },0,x} \leq u(0,x),$ using \eqref{eq1bis}, we get
$
Y_{0, \theta}^{\alpha,0,x}\left[\bar u _*(\theta,X_{\theta}^{\alpha,0,x})\right] \leq Y_{0, \theta^n}^{\alpha,0,x}[ \bar u _* (\theta^n,X_{\theta^n}^{\alpha,0,x}) ] 
\leq u(0,x) + K \varepsilon .
$
 Taking the supremum on $\alpha \in \mathcal{A}$ and letting $ \varepsilon$ tend to $0$, we obtain inequality \eqref{sense2}.

It remains to show  the first assertion.  It is sufficient to show that for each $\theta \in \mathcal{T}$,
\begin{equation}\label{sense1}
u(0,x) \leq {\sup_{\alpha \in \mathcal{A}}}\, Y_{0,\theta}^{\alpha,0,x}\left[\bar u ^*(\theta,X_{\theta}^{\alpha,0,x})\right] \text{   }.
\end{equation}
Let $\theta \in \mathcal{T}$. Let $\alpha \in \mathcal{A}$.
As above, we approximate $\theta$ by the sequence of stopping times $(\theta^n)_{n\in \mathbb N}$ defined above. Let $n\in \mathbb N$.
By applying the flow property for reflected BSDEs, we get 
$
Y_{0,T}^{\alpha,0,x}= 
Y_{0,\theta^n}^{\alpha,0,x}[Y_{\theta^n,T}^{\alpha,\theta^n,X_{\theta^n}^{\alpha,0,x}}].
$
By similar arguments as in the proof of the super--optimality principle (but without using the existence 
of  $\varepsilon$-optimal controls), we derive that
$Y_{\theta^n,T}^{\alpha,\theta^n,X_{\theta^n}^{\alpha,0,x}} \leq \bar u ^*(\theta^n,X_{\theta^n}^{\alpha,0,x})$ 
a.s.\, By the comparison theorem for reflected BSDEs, it follows that
$$ Y_{0,T}^{\alpha,0,x}= 
Y_{0,\theta^n}^{\alpha,0,x}[Y_{\theta^n,T}^{\alpha,\theta^n,X_{\theta^n}^{\alpha,0,x}}]  \leq  Y_{0,\theta^n}^{\alpha,0,x}[\bar u ^*(\theta^n,X_{\theta^n}^{\alpha,0,x})].$$
Using the Fatou lemma for reflected BSDEs (Prop. \ref{fatou2}), we get:
$$ Y_{0,T}^{\alpha,0,x} \leq \lim \sup_{n \rightarrow \infty} Y_{0,\theta^n}^{\alpha,0,x}[\bar u ^*(\theta^n,X_{\theta^n}^{\alpha,0,x})]  \leq Y_{0,\theta}^{\alpha,0,x}[\lim \sup_{n \rightarrow \infty} \bar u ^*(\theta^n,X_{\theta^n}^{\alpha,0,x})].$$
Using the upper semicontinuity property of $\bar u ^*$ on $[0,T[ \times \R$ and  $\bar u ^*(T,x)= g(x)$, we  obtain
$$Y_{0,T}^{\alpha,0,x} \leq  Y_{0,\theta}^{\alpha,0,x}[\lim \sup_{n \rightarrow \infty} \bar u ^*(\theta^n,X_{\theta^n}^{\alpha,0,x})] \leq Y_{0,\theta}^{\alpha,0,x}[\bar u ^*(\theta,X_{\theta}^{\alpha,0,x})].$$

\noindent Since  $\alpha \in \mathcal{A}$ is arbitrary, we get inequality \eqref{sense1}, which completes the proof.
\fproof

\section{Nonlinear HJB variational inequalities}\label{sec4}

\subsection{Some extensions of comparison theorems for BSDEs and  reflected BSDEs}
We provide two results which will be used to prove that the value function $u$, defined by \eqref{probform1}, is a {\em weak} viscosity solution of some nonlinear Hamilton Jacobi Bellman variational inequalities (see Theorem \ref{existprob}). 
We first show a slight extension of the comparison theorem for BSDEs given in  \cite{16}, from which we derive a comparison result between a BSDE and a reflected BSDE.
\begin{lemma}\label{A.4}
Let $t_0 \in [0,T]$ and let $\theta \in \mathcal{T}_{t_0}$. 
Let $\xi_1$ and $\xi_2$ $\in$ $L^2({\cal F}_{\theta})$.
Let $f_1$ be  a driver.
Let $f_2$ be a Lipschitz driver with Lipschitz constant $C>0$, satisfying Assumption \ref{Royer}. 
For $i=1,2$, let $(X^i_t, \pi^i_t, l^i_t)$ be a solution in $\mathcal{S}^{2} \times \H^{2} \times \H_{\nu}^{2}$ of the BSDE associated with driver $f_i$, terminal time $\theta$ and terminal condition $\xi_i$.
Suppose 
 that
$$ 
f_1 (t, X^1_t, \pi^1_t, l^1_t) \geq f_2 (t, X^1_t, \pi^1_t, l^1_t)  \;\; t_0 \leq t \leq \theta,\; \; dt\otimes dP \text{ a.s.} \quad {\rm and} \quad \xi_1 \geq \xi_2  + \vp  \text{ a.s. }
$$
 where $\varepsilon $ is a real constant.  
 Then, for each $t \in [t_0, \theta]$, we have $X_{t}^1 \geq X_{t}^2 + \varepsilon \, e^{-CT}$ a.s.
  \end{lemma}
 \begin{proof}
 From inequality (4.22) in the proof of the Comparison Theorem in \cite{16}, we derive that 
 $
  X_{t_0}^1-X_{t_0}^2 \geq  e^{-CT} {\mathbb E} \left[  H_{t_0, \theta}\,  \varepsilon \, | \mathcal{F}_{t_0} \right]\;
 $ a.s.\,,
where $C$ is the Lipschitz constant of $f_2$, and 
$(H_{t_0,s})_{s \in [t_0,T]}$ is the
 non negative martingale satisfying
$d H_{t_0,s}  = H_{t_0,s^-} [ \beta_s d W_s + \int_{{\bf E}} \gamma_s(u) \Tilde{N}(ds,du)]$ with 
$H_{t_0,t_0}  = 1$,
 $(\beta_s)$ being  a predictable process bounded by $C$. The result follows.
 \end{proof}

\begin{proposition}[A comparison result between a BSDE and a reflected BSDE]\label{ecart}
Let $t_0 \in [0,T]$ and let $\theta \in \mathcal{T}_{t_0}$. 
Let $\xi_1$ $\in$ $L^2({\cal F}_{\theta})$ and
let $f_1$ be  a driver.
  Let $(X_t^{1}, \pi_t^1, l_t^1)$ be a solution of the  BSDE associated with $f_1$, terminal time $\theta$ and terminal condition $\xi^1$. Let $(\xi_t^2)$ $\in$ $\mathcal{S}^2$ and let $f_2$ be a Lipschitz driver with Lipschitz constant $C>0$ which satisfies Assumption \ref{Royer}.
 Let $(Y^2_t)$ be the solution of the reflected BSDE associated with $f_2$, terminal time $\theta$ and obstacle $(\xi_t^2)$. 
Suppose that
\begin{equation}\label{autre1}
f_1(t, X^1_t, \pi^1_t, l^1_t) \geq f_2(t, X_t^1,\pi_t^1,l^1_t), \;\; t_0 \leq t \leq \theta,\;  dt\otimes dP \text{-a.s.}\,\,{\rm and}\,\,
X_t^1 \geq \xi_t^2+\varepsilon, \text{  } t_0 \leq t \leq \theta \text{ a.s.}
\end{equation}
Then, we have $X_t^1 \geq Y^2_t+\varepsilon e^{-CT},\, $  $ t_0 \leq t \leq \theta $ a.s.

\end{proposition}

\begin{proof} Let $t \in [t_0, \theta]$.
By the characterization of the solution of the RBSDE as the value function of an optimal stopping problem (see Th. 3.2 in \cite{16}), $Y_t^2= ess \sup_{\tau \in \mathcal{T}_{[ t, \theta]}} \mathcal{E}_{t, \tau}^{f^2}(\xi_\tau^2)$. By Lemma \ref{A.4}, for each $\tau \in \mathcal{T}_{[ t, \theta]}$, $X_t^1 \geq \mathcal{E}_{t, \tau}^{f^2}(\xi_\tau^2)+e^{-CT} \varepsilon$. Taking the supremum over $\tau \in \mathcal{T}_{[ t, \theta]}$, the result follows.
\end{proof}


%

\subsection{Links between the mixed control problem and HJB equation}

We introduce the  following  Hamilton Jacobi Bellman variational inequality  (HJBVI):
\begin{equation}\label{edp}
\begin{cases}
 \min(u(t,x)-h(t,x), \\\quad \inf_{\alpha \in {\bf {\bf A}}}(-\dfrac{\p u}{\p t}(t,x)-L^{\alpha}u(t,x)-f(\alpha,t,x,u(t,x), (\sigma \dfrac{\p u }{\p x})(t,x),B^{\alpha}u(t,x)) )=0, (t,x) \in [0,T) \times \mathbb{R}\\
u(T,x)=g(x), x\in \mathbb{R}
\end{cases}
\end{equation}
where $L ^{\alpha}:=A^{\alpha}+K^{\alpha},$ and for $\phi \in C^2(\mathbb{R})$,
\begin{itemize}

\item
$A^{\alpha}\phi(x) := \dfrac{1}{2}\sigma^2(x,\alpha)\dfrac{\p^2 \phi}{\p x^2}(x)+ 
b(x,\alpha) \dfrac{\p \phi }{\p x}(x)$ and $B^{\alpha} \phi(x) :=\phi(x+\beta(x,\alpha,\cdot))-\phi(x).$
\item
$K^{\alpha} \phi(x) :=\int_{{\bf E}}\left(\phi(x+ \beta(x,\alpha,e))-\phi(x)- \dfrac{\p \phi}{\p x}(x)\beta(x,\alpha,e)\right) \nu (de)$.
\end{itemize}
\begin{definition}\rm
$\bullet$ A function $u$ is said to be a {\em viscosity subsolution} of \eqref{edp} if it is upper semicontinuous    on $[0,T] \times \mathbb{R}$, 
and if for any point $(t_0,x_0) \in [0,T[ \times \mathbb{R}$ and for any $\phi \in C^{1,2}([0,T] \times \mathbb{R})$ such that $\phi(t_0,x_0)=u(t_0,x_0)$ and $\phi-u$ attains its minimum at $(t_0,x_0)$, we have
\begin{align}\label{equation4.3}
&\min(u(t_0,x_0)-h(t_0,x_0),\nonumber\\ &\inf_{\alpha \in {\bf A}}(-\dfrac{\p \phi}{\p t}(t_0,x_0)-L^{\alpha}\phi(t_0,x_0)-f(\alpha,t_0,x_0,u(t_0,x_0), (\sigma \dfrac{\p \phi}{\p x})(t_0,x_0),B^{\alpha} \phi(t_0,x_0))) \leq 0.
\end{align}
In other words, if $u(t_0,x_0)>h(t_0,x_0)$, then
\begin{equation*}
\inf_{\alpha \in  {\bf A}}(-\dfrac{\p \phi}{\p t} (t_0,x_0)-L^{\alpha}\phi(t_0,x_0)-f(\alpha,t_0,x_0,u(t_0,x_0),  (\sigma \dfrac{\p \phi}{\p x} )(t_0,x_0),B^{\alpha} \phi(t_0,x_0))) \leq 0.
\end{equation*}

$\bullet$
 A  function $u$  is said to be a {\em  viscosity supersolution } of \eqref{edp} if  it is lower semicontinuous    on $[0,T] \times \mathbb{R}$, 
 and if for any point $(t_0,x_0) \in [0,T[ \times \mathbb{R}$ and any $\phi \in C^{1,2}([0,T] \times \mathbb{R})$ such that $\phi(t_0,x_0)=u(t_0,x_0)$ and $\phi-u$ attains its maximum at $(t_0,x_0)$, we have
\begin{align*}
&\min(u(t_0,x_0)-h(t_0,x_0),\\& \inf_{\alpha \in {\bf A}}(-\dfrac{\p}{\p t} \phi(t_0,x_0)-L^{\alpha} \phi(t_0,x_0)-f(\alpha,t_0,x_0,u(t_0,x_0),  (\sigma \dfrac{\p \phi}{\p x})(t_0,x_0),B^{\alpha} \phi (t_0,x_0))) \geq 0.
\end{align*}
In other words, we have both $u(t_0,x_0) \geq h(t_0,x_0)$ and
\begin{equation}\label{supermontrer}
\inf_{\alpha \in {\bf A}}(-\dfrac{\p \phi}{\p t}(t_0,x_0)-L^{\alpha} \phi(t_0,x_0)-f(\alpha,t_0,x_0,u(t_0,x_0), (\sigma \dfrac{\p \phi}{\p x} )(t_0,x_0),B^{\alpha} \phi(t_0,x_0))) \geq 0.
\end{equation}
\end{definition}




 Using the {\em weak} dynamic programming principle given in Theorem \ref{IMP} and Proposition \ref{ecart}, we now 
 prove 
that the value function of our problem is a {\em weak} viscosity solution of the above HJBVI.

\begin{theorem}\label{existprob} The value function $u$, defined by \eqref{probform1}, is a {\em weak viscosity solution} of the  HJBVI  \eqref{edp}, in the sense that 
its  u.s.c. envelope $u^*$ 
is a viscosity subsolution of  \eqref{edp} and its l.s.c. envelope  $u_*$  is a viscosity supersolution of  \eqref{edp} (with terminal condition $u(T,x) = g(x)$).

\end{theorem}

\dproof
$\bullet$    {\em We first prove that $u^*$  is a subsolution of \eqref{edp}.} 
Let $(t_0,x_0) \in [0,T[ \times \mathbb{R}$ and $\phi \in C^{1,2}([0,T] \times \mathbb{R})$ be such that $\phi(t_0, x_0)=u^*(t_0,x_0)$  and $\phi(t,x) \geq u^*(t,x)$, $\forall (t,x) \in [0,T] \times \mathbb{R}$.
Without loss of generality, we can suppose that the minimum of $u^*-\phi$ attained at $(t_0,x_0)$ is strict. 
Suppose for contradiction that $u^*(t_0,x_0)>h(t_0,x_0)$ and that
\begin{equation*}
\inf_{\alpha \in {\bf A}} (-\dfrac{\p}{\p t} \phi(t_0,x_0)-L^{\alpha}\phi(t_0,x_0)-f(\alpha,t_0,x_0, \phi (t_0,x_0), (\sigma \dfrac{\p \phi}{\p x})(t_0,x_0),B^{\alpha} \phi(t_0,x_0)))>0.
\end{equation*}
By uniform continuity of $K^{\alpha} \phi$ 
 and $B^{\alpha} \phi:[0,T] \times \mathbb{R} \rightarrow {L}_{\nu}^2$ with respect to $\alpha$, we can suppose that there exists $\epsilon>0$ , $\eta_\epsilon>0$ such that:
$\forall (t,x)$ such that $t_0 \leq t \leq t_0+ \eta_\epsilon<T$ and $|x-x_0| \leq \eta_\epsilon$, we have: $\phi(t,x) \geq h(t,x) + \epsilon$ and
\begin{equation}\label{bulle}
-\dfrac{\p}{\p t}\phi(t,x)-L^{\alpha} \phi(t,x)-f(\alpha,t,x, \phi (t,x), (\sigma \dfrac{\p \phi}{\p x})(t,x),B^{\alpha} \phi(t,x)) \geq \epsilon, \text{  } \forall \alpha \in {\bf A}.
\end{equation}
 We denote by $B_{\eta_\varepsilon}(t_0,x_0)$  the ball of radius $\eta_\varepsilon$ and center $(t_0,x_0)$. By definition of $u^*$, there exists a sequence $(t_n,x_n)_n$ in $B_{\eta_\varepsilon}(t_0,x_0)$, such that $(t_n,x_n,u(t_n,x_n)) \rightarrow (t_0,x_0,u^*(t_0,x_0))$.\\
  Fix $n \in \mathbb{N}$. 
 Let $\alpha$ be  an arbitrary control of $\mathcal{A}_{t_n}^{t_n}$ and $X^{\alpha,t_n,x_n}$ the associated state process.\\
We define the stopping time $\theta^{\alpha,n}$ as
\begin{equation*}\label{4.10}
\theta^{\alpha, n}:= (t_0 + \eta_\epsilon) \wedge \inf\{s \geq t_n\,,\, |X_s^{\alpha,t_n,x_n}-x_0| \geq \eta_\epsilon \}.
\end{equation*}
Let  $\psi^{\alpha}(s,x):=\dfrac{\p}{\p s}\phi(s,x)+L^{\alpha}\phi(s,x)$.
Applying It\^o's lemma to $\phi(t, X_{t}^{\alpha,{t_n},x_n})$, we derive that 
\begin{equation*}
(\phi(s,X_s^{\alpha,t_n,x_n}), (\sigma \dfrac{\p \phi }{\p x})(s, X_s^{\alpha,t_n,x_n}), B^{\alpha_s}\phi(s,X_{s^-}^{\alpha,t_n,x_n}); s \in[t_n,\theta^{\alpha,n}])
\end{equation*}
 is the solution of the BSDE associated with the driver process  $-\psi^{\alpha_s}(s,X_s^{\alpha,t_n,x_n})$,   terminal time $\theta^{\alpha,n}$ and terminal value $\phi(\theta^{\alpha,n}, X_{\theta^{\alpha,n}}^{\alpha,t_n,x_n})$. 
By \eqref{bulle} and by definition of  $\theta^{\alpha,n}$, we get 
 \begin{equation}\label{4.13}
-\psi^{\alpha_s}(s,X_s^{\alpha,t_n,x_n}) \geq f (\alpha_s,s,X_s^{\alpha,t_n,x_n},\phi(s,X_s^{\alpha,t_n,x_n}),(\sigma \dfrac{\p \phi}{\p x})(s,X_s^{\alpha,t_n,x_n}),B \phi(s,X_s^{\alpha,t_n,x_n})) + \epsilon
\end{equation}
for each $s \in [t_n,\theta^{\alpha,n}]$. This inequality gives a relation between the drivers  $-\psi^{\alpha_s}(s,X_s^{\alpha,t_n,x_n})$ and 
$f(\alpha_s, \cdot)$ of two BSDEs.
 Now, since the minimum $(t_0,x_0)$ is strict, there exists $\gamma_{\epsilon}$
 such that: 
\begin{equation}\label{wifi}
u^*(t,x)-\phi(t,x) \leq -\gamma_\epsilon \text{  on  }  [0,T] \times \mathbb{R} \setminus B_{\eta_\epsilon}(t_0,x_0).
\end{equation}
We have
\begin{equation*}\label{wifi1}
\phi(\theta^{\alpha,n}\wedge t, X_{\theta^{\alpha,n} \wedge t}^{\alpha,t_n,x_n})=\phi(t, X_{t}^{\alpha,t_n,x_n})\textbf{1}_{t<\theta^{\alpha,n}}+\phi(\theta^{\alpha,n}, X_{\theta^{\alpha,n}}^{\alpha,t_n,x_n})\textbf{1}_{t \geq\theta^{\alpha,n}}, \,\,\, t_n \leq t \leq T \quad{\rm a.s.}
\end{equation*}
To simplify notation, set $\delta_\varepsilon:=\min(\epsilon, \gamma_{\epsilon})$. Using  \eqref{wifi} together with the definition of $\theta^{\alpha,n}$, we get 
$$
\phi( t, X_{ t}^{\alpha,t_n,x_n}) \geq (h(t, X_{t}^{\alpha,t_n,x_n})+\delta_\varepsilon)\textbf{1}_{t<\theta^{\alpha,n}}+(u^*(\theta^{\alpha,n}, X_{\theta^{\alpha,n}}^{\alpha,t_n,x_n})+\delta_\varepsilon)\textbf{1}_{t = \theta^{\alpha,n}},
 \quad  t_n \leq t\leq \theta^{\alpha,n} \quad{\rm a.s.}
$$
This, together with inequality \eqref{4.13} on the drivers and the above comparison theorem between a BSDE and a reflected BSDE  (see Proposition \ref{ecart}) lead to:
\begin{equation*}
\phi(t_n,x_n) \geq Y_{t_n, \theta^{\alpha,n}}^{\alpha,t_n,x_n}[h(t, X_{t}^{\alpha,t_n,x_n})\textbf{1}_{t<\theta^{\alpha,n}}+u^*(\theta^{\alpha,n}, X_{\theta^{\alpha,n}}^{\alpha,t_n,x_n})\textbf{1}_{t=\theta^\alpha}]+\delta_\varepsilon K, 
\end{equation*}
where  $K$ is a positive constant which only depends on $T$  and the Lipschitz constant of $f$.\\
Now, recall $(t_n,x_n,u(t_n,x_n)) \rightarrow (t_0,x_0,u^*(t_0,x_0))$ and $\phi$ is continuous 
with $\phi(t_0,x_0)=u^*(t_0,x_0)$.
We can thus assume that $n$ is sufficiently large so that
$
|\phi(t_n,x_n)-u(t_n,x_n)| \leq \delta_\varepsilon K/{2}.
$
Hence, 
\begin{equation*}\label{contract}
u(t_n,x_n) \geq Y_{t_n, \theta^{\alpha,n}}^{\alpha,t_n,x_n}[h(t, X_{t}^{\alpha,t_n,x_n})\textbf{1}_{t<\theta^{\alpha,n}}+u^*(\theta^{\alpha,n}, X_{\theta^{\alpha,n}}^{\alpha,t_n,x_n})\textbf{1}_{t=\theta^\alpha}]+
\delta_\varepsilon K/{2}.
\end{equation*}
As this inequality holds for all $\alpha \in \mathcal{A}_{t_n}^{t_n}$ and 
since $u^* \geq \bar u^*$, we get a contradiction of  the sub-optimality principle of dynamic programming principle \eqref{DPP}  (see also Remark \ref{R}).

\bigskip

$\bullet$ {\em  We now prove that $u_*$  is a viscosity supersolution of \eqref{edp}.}

Let $(t_0,x_0) \in [0,T[ \times \mathbb{R}$ and $\phi \in C^{1,2}([0,T] \times \mathbb{R})$ be such that $\phi(t_0, x_0)=u_*(t_0,x_0)$  and $\phi(t,x) \leq u_*(t,x)$, $\forall (t,x) \in [0,T] \times \mathbb{R}$. 
Without loss of generality, we can suppose that the maximum is strict in $(t_0,x_0)$.
Since the solution $(Y_s^{\alpha,t_0,x_0})$ stays above the obstacle, for each $\alpha \in \mathcal{A}$, 
we have
$
u_*(t_0,x_0) \geq h(t_0,x_0).
$
Our 	aim is to show that inequality \eqref{supermontrer} holds.\\
Suppose for contradiction that this inequality does not hold.\\
By continuity, we can suppose that there exists $\alpha \in {\bf A}$, $\epsilon>0$ and $\eta_\epsilon>0$ such that:\\
$\forall (t,x)$ with $t_0 \leq t \leq t_0+ \eta_\epsilon<T$ and $|x-x_0| \leq \eta_\epsilon$, we have:
\begin{equation}\label{boule1}
 -\dfrac{\p}{\p t} \phi(t,x)-L^\alpha \phi(t,x)-f(\alpha,t,x,\phi(t,x), (\sigma \dfrac{\p \phi }{\p x})(t,x),B^\alpha \phi(t,x)) \leq -\epsilon.
\end{equation}
We denote by $B_{\eta_\varepsilon}(t_0,x_0)$  the ball of radius $\eta_\varepsilon$ and center $(t_0,x_0)$. Let $(t_n,x_n)_n$ be a sequence in $B_{\eta_\varepsilon}(t_0,x_0)$  such that $(t_n,x_n,u(t_n,x_n)) \rightarrow (t_0,x_0,u_*(t_0,x_0))$.
We introduce the state process $X^{\alpha,t_n,x_n}$ associated with the above constant control $\alpha$ and define  the stopping time $\theta^n$ as:
\begin{equation*}
\theta^n:= (t_0 + \eta_\epsilon) \wedge \inf\{s \geq t_n \, , \, |X_s^{\alpha,t_n,x_n}-x_0| \geq \eta_\epsilon \}.
\end{equation*}
By It\^o's formula, the process 
$(\phi(s,X_s^{\alpha,t_n,x_n}), (\sigma \dfrac{\p \phi}{\p x})(s, X_s^{\alpha,t_n,x_n}), B^\alpha\phi(s,X_{s^-}^{\alpha,t_n,x_n}); s \in [t_n, \theta^n])$
 is the solution of the BSDE associated with 
terminal time $\theta^n$,  terminal value $\phi(\theta^n, X_{\theta^n}^{\alpha,t_n,x_n})$ and driver $-\psi^\alpha(s,X_s^{\alpha,t_n,x_n})$. 
The definition of the stopping time $\theta^n$ and inequality \eqref{boule1} lead to:
\begin{align}\label{eqq}
&-\psi^{\alpha}(s,X_s^{\alpha,t_n,x_n}) \leq f(\alpha,s,X_s^{\alpha,t_n,x_n},\phi (s,X_s^{\alpha,t_n,x_n}),  (\sigma \dfrac{\p \phi}{\p x} )(s,X_s^{\alpha,t_n,x_n}),B^\alpha \phi (s,X_s^{\alpha,t_n,x_n})),
\end{align}
for $t_n \leq s \leq \theta^n$ $ds \otimes dP$-a.s.\,
 Now, since the maximum $(t_0,x_0)$ is strict, there exists $\gamma_\epsilon$ (which depends on $\eta_\epsilon$) such that 
$
u_*(t,x) \geq \phi(t,x) +\gamma_\epsilon \text{  on  }  [0,T] \times \mathbb{R} \setminus B_{\eta_\epsilon}(t_0,x_0)
$
which implies
$\phi(\theta^n, X_{\theta^n}^{\alpha,t_n,x_n}) \leq u_*(\theta^n,X_{\theta^n}^{\alpha,t_n,x_n})-\gamma_\epsilon.$
Hence, using inequality  \eqref{eqq} on the drivers, 
together with  the comparison theorem for BSDEs, we derive that:
\begin{equation*}\label{CPPD}
\phi(t_n,x_n) = \mathcal{E}_{t_n, \theta^n}^{-\psi^{\alpha}}[\phi(\theta^n,X_{\theta^n}^{\alpha,t_n,x_n})]\leq \mathcal{E}_{t_n, \theta^n}^{\alpha,t_n,x_n}[u_*(\theta^n,X_{\theta}^{\alpha,t_n,x_n})-\gamma_\epsilon] \leq \mathcal{E}_{t_n,  \theta^n}^{\alpha,t_n,x_n}[u_*(\theta^n,X_{\theta^n}^{\alpha,t_n,x_n})]-\gamma_\epsilon K. 
\end{equation*}
where the second inequality follows from an extension of the comparison theorem (Lemma \ref{A.4}).
We can assume that  $n$ is sufficient large so that
$
|\phi(t_n,x_n)-u(t_n,x_n)| \leq \delta_\varepsilon K/{2}.
$
We thus get:
\begin{equation}\label{CPPD1}
u(t_n,x_n) \leq \mathcal{E}_{t_n,  \theta^n}^{\alpha,t_n,x_n}[u_*(\theta^n,X_{\theta^n}^{\alpha,t_n,x_n})]-
{\gamma_\epsilon K}/{2}. 
\end{equation}
Since $u$ satisfies the super-optimality DPP (Th. \ref{IMP}), we have
$
u(t_n,x_n) \geq \mathcal{E}_{t_n,\theta^n}^{\alpha,t_n,x_n}[\bar u_*(\theta^n,X_{\theta^n}^{\alpha,t_n,x_n})].
$ Since $\bar u_* \geq u_*$, this inequality with \eqref{CPPD1} leads to a contradiction. 
\fproof
\begin{remark}\label{caracfaible} When $g$ is only Borelian, 
the {\em weak} solution of the HJB equation \eqref{edp}
is generally not unique, even in the deterministic case (as stressed in \cite{BJ, B,BP}). 

Note that when $g$ is l.s.c., the value function $u$ of our problem can be shown to  be the minimal (l.s.c.) viscosity supersolution of the HJB equation \eqref{edp}, with terminal value greater than $g$
 (by using similar arguments as in the proof of Th.6.5 in \cite{DQS2}).

Note also that  
the paper \cite{BJ}  (see also \cite{B}) provides a characterization of the 
u.s.c. envelope $u^*$ of the value function of the deterministic control problem  
 $ u(t,x):=\sup_{\alpha \in {\cal A}_t} g(X_T^{t,x, \alpha})$, which corresponds to our problem with $\sigma=f=0$ and no stopping times controls.
 More precisely, when $g$ is {\em u.s.c.}, the map 
$u^*$ is characterized as the unique  u.s.c. viscosity solution of the HJB equation (i.e. satisfies  $u^*(T,x) = g(x)$ and the analogous of \eqref{equation4.3} but with an equality). 
The
proof is based on PDEs arguments and deterministic control theory.
 An interesting  further development of our paper 
(and of \cite{BT}) would be to 
study analogous properties in the stochastic case. 

\end{remark}

\appendix
\section{Appendix}
We give here some measurability results which are used in Section \ref{mesusection}.  We start by 
 the proof of  Proposition \ref{gborelian}. To this purpose, we first provide the following  lemma:   \begin{lemma}\label{limite}
  Let $(\Omega, {\cal F}, P)$ be a probability space. 
Suppose that the Hilbert space $L^2:= L^{2}(\Omega, {\cal F}, P)$ equipped with the usual scalar product is separable.
Let $F$ $\in L^2$.\\
Consider a sequence of functions $(g_n)_{n \in  \mathbb{N}}$ such that for each $n$, 
$g_n: \mathbb{R} \rightarrow  \mathbb{R}$ is Borelian, with $| g_n(x)| \leq C(1+ |x|^p)$.
 Suppose that sequence $(g_n)_{n \in  \mathbb{N}}$ converges pointwise.\\
  Let $g$ be the limit, defined for each 
 $x \in  \mathbb{R}$ by $g(x) := \lim_{n \rightarrow + \infty} g_n(x)$. 
 
Suppose also that for each $n \in  \mathbb{N}$, the map $\psi ^{g_n,F}$ (denoted also by $\psi ^{g_n}$) defined by 
 $
 \psi ^{g_n}: L^{2p} \cap L^2 \rightarrow \mathbb{R}; \,\,  \xi \mapsto \mathbb{E}[g_n(\xi)\,F]
$
  is Borelian,   
  $L^{2p} \cap L^2$ being equipped with the $\sigma$-algebra 
 induced by  ${\cal B}(L^2)$.
 
Then, the map $\psi ^{g}$ (denoted also by $\psi ^{g, F}$) defined by 
 \begin{equation}\label{phig}
 \psi ^{g}: L^{2p} \cap L^2 \rightarrow \mathbb{R}; \,\,   \xi \mapsto \mathbb{E}[g(\xi)\,F]
 \end{equation}
  is Borelian. 
\end{lemma}

\dproof
By the Lebesgue theorem, for each $\xi \in L^{2p} \cap L^2$, we have 
$\psi ^{g}(\xi) \,= \, \mathbb{E}[g(\xi)\,F] = \lim_{n \rightarrow + \infty}
\mathbb{E}[g_n(\xi)\,F] \,= \lim_{n \rightarrow + \infty}\psi ^{g_n} (\xi).$
Since
 the pointwise limit of a sequence of $\mathbb{R}$-valued measurable maps is measurable, we derive that the map
$\psi ^{g}$ is Borelian.
\fproof

%
{\bf Proof of  Proposition \ref{gborelian}}.  Since by assumption the Hilbert space $L^2$ is separable,  there exists a countable orthonormal basis $\{e^{i}, i \in \mathbb{N} \}$ of $L^2$. 
For each $\xi \in L^{2p} \cap L^2$, we have 
 $\varphi^g(\xi)= g(\xi)= \sum_{i}\psi^{g,i}(\xi)\,
e_i$ in $L^2$, where   $\psi^{g,i}(\xi):= \mathbb{E}[g(\xi)\,e_i]$ for each $i \in \mathbb{N}$. 
 Hence, in order to show the measurability of the map $\psi$, it is sufficient to show the measurability of  the maps $\psi ^{g, F}$, $F$ $\in$ $L^2$.


To this purpose, we introduce the set ${\cal H}$ of bounded Borelian functions $g: \mathbb{R} \rightarrow  \mathbb{R}$ such that for each 
$F$ $\in$ $L^2$, the map $\psi ^{g, F}$ is Borelian. Note that ${\cal H}$ is a vector space. 
Suppose we have shown that for all real numbers $a,b$ with $a < b$, 
 ${\bf 1} _{]a,b[} \in {\cal H}$.
Then, by Lemmas \ref{limite} together with a monotone class theorem, 
 we derive that 
${\cal H}$ is equal to the whole set of bounded Borelian functions.
When $g$ is not bounded, the result follows by approximating $g$ by a sequence of bounded Borelian functions, and by using  Lemma \ref{limite}.

It remains to show that for all $a,b \in {\mathbb R}$ with $a < b$, we have 
 ${\bf 1} _{]a,b[} \in {\cal H}$.
Since ${\bf 1} _{]a,b[}$ is l.s.c., it follows that 
there exists a non decreasing sequence $(g_n)_{n \in  \mathbb{N}}$ of Lipschitz continuous functions (taking their values in $[0,1]$) such that for each 
 $x \in  \mathbb{R}$, ${\bf 1} _{]a,b[}(x) := \lim_{n \rightarrow + \infty} g_n(x)$.  For each $n$, since $g_n$  is Lipschitz continuous, by using Cauchy Schwartz's inequality, one can derive that the map $\psi ^{g_n}:   L^2 \rightarrow \mathbb{R}; \,\,  \xi \mapsto \mathbb{E}[g_n(\xi)\,F]
 $ is Lipschitz continuous for the norm $\| \cdot \|_{L^2}$,
 and hence Borelian.  The result then follows from 
 Lemma \ref{limite}.
 \fproof

We now state the following lemma, which is needed in the proof of Theorem \ref{mesu}.
\begin{lemma}\label{sepa} Let $t \in [0,T]$.
The Hilbert space $L^2( \Omega, {\cal F}^t_T, P)$ (simply denoted by $L_t^2$)
 is separable. 
Moreover, the Hilbert space $\mathbb{H}_t^2$ is separable. 
\end{lemma}
\dproof The proof is given for completeness.
 Recall first that, given a  probabilistic space $(\Omega,{\cal B},P)$,  if ${\cal B}$ is 
  countably generated, then $L^2(\Omega,{\cal B},P )$  is separable (see e.g. Proposition 3.4.5 in \cite{C}).
In our case,  $\Omega$ is separable, which implies that the Borelian $\sigma$-algebra $\mathcal{B}(\Omega)$  is   countably generated. The space  $L^2( \Omega, \mathcal{B}(\Omega), P)$ is thus separable. 
Let now $t$ be any real in $[0,T]$.  
We introduce ${\mathbb F}^{o,t}=({\cal F}_s^{o,t})_{s\geq t}$ the natural filtration of $W^t$ and $N^t$. By definition, ${\mathbb F}^t$ is 
the completed filtration of ${\mathbb F}^{o,t}$ (with respect to 
${\cal B}(\Omega)$ and $P$).
 For each $\xi$ $\in$ $L_t^2= L^2(\Omega, {\cal F}^t_T,P)$, there exists an  ${\cal F}_T^{o,t}$-measurable random variable $\xi '$ such that $\xi= \xi '$ $P$-a.s. 
 Hence, $L_t^2$ can be identified with $L^2( \Omega, {\cal F}^{o,t}_T,P)$,
 which is separable because ${\cal F}^{o,t}_T= (T^t)^{-1} ( \mathcal{B}(\Omega))$ is countably generated.\\
Now,  denote by  $\mathcal{P}^{o,t}$ the predictable $\sigma$-algebra associated with ${\mathbb F}^{o,t}$.
For each $\mathcal{P}^t$-measurable process $(X_s)$, there exists a $\mathcal{P}^{o,t}$-measurable process $(X'_s)$ indistinguishable of $(X_s)$ (see \cite{DM1}   IV \textsection 79 or \cite{J} I Prop.1.1 p.8). 
Hence, the space $\mathbb{H}_t^2$$=$ $L^2( [t,T] \times \Omega, \mathcal{P}^t, ds \otimes  dP)$ can be identified with the Hilbert space $L^2( [t,T] \times \Omega, \mathcal{P}^{o,t}, ds \otimes  dP)$.
Since the paths are right-continuous, for every $r>t$, $\mathcal{F}^{o,t}_{r^-}= 
\sigma ( \{\omega^t_u\,,\, u \in \mathbb{Q}\,\, {\rm and} \,\,  t \leq u <r\})$ and is thus  countably generated.
The predictable $\sigma$-algebra $\mathcal{P}^{o,t}$  is generated by the sets of the form $[r, T[ \times H$ 
(or $]r, T] \times H$), where $r$ is rational with $r\geq t$, and $H$ belongs to $\mathcal{F}^{o,t}_{r^-}$. It follows that $\mathcal{P}^{o,t}$ is  countably generated.
  Hence,  $ L^2( [t,T] \times \Omega, \mathcal{P}^{o,t}, ds \otimes  dP)$ is separable, which gives that $\mathbb{H}_t^2$ is separable.
\fproof

\noindent  Lemma 1.2 in \cite{Crauel} ensures the following property which is used in the proof of Theorem \ref{select}.
\begin{lemma}[A result of Measure Theory] \label{Cra}
Let $(X, {\cal F}, Q)$ be a probability space. Let ${\cal F}_Q$ be the {\em completion $\sigma$-algebra of ${\cal F}$ with respect to $Q$}, 
 that is the class of sets of the form $B \cup M$, with $B$ $\in$ $\mathcal{F} $
 and  $M$ being a {\em $Q$-null set}, that is a subset of a set  $N$ belonging to $\mathcal{F} $ with $Q$-measure $0$. 
Let $E$ be a separable Hilbert space, equipped with its scalar product $< .\,, \,.>$,  and its Borel $\sigma$-algebra ${\cal B} (E)$.\\
 Then, for each  ${\cal F}_Q$-measurable map $f: \,X \rightarrow E$, there exists an ${\cal F}$-measurable map $f_{Q}$ such that  
$f_Q (x) =  f(x)$ for $Q$-{\em almost every} $x$, in the sense that the set $\{ x \in X\,, \,f_Q (x) \neq  f(x)\}$ is included in a set  belonging to ${\cal F}$  with $Q$-measure $0$.
\end{lemma}

 \paragraph{A result of classical analysis} (used  in the proof of Lemma \ref{zz}).\\
  For each $n \in \mathbb N$, we consider the linear operator $P^n: L^2([0,T], dr) \rightarrow L^2([0,T], dr)$ 
 defined for each $f \in L^2([0,T], dr)$ by 
 $P^n(f)(t):=n \sum_{i=1}^{n-1}( \int_{\frac{(i-1)T}{n}}^{\frac{iT}{n}}f(r) dr) \textbf{1}_{]\frac{i T}{n}, \frac{(i+1)T}{n}]}(t). $
By Cauchy-Schwartz's inequality, we have that for each $t\in ]\frac{i T}{n}, \frac{(i+1)T}{n}]$, $1 \leq i \leq n-1$, $P^n(f)^2(t) \leq n  \int_{\frac{(i-1)T}{n}}^{\frac{iT}{n}}f^2(r) dr$. Hence,
\begin{equation}\label{Pn}
||P_n(f)||_{L^2_T}\leq ||f||_{L^2_T}; \,\,\,\quad ||P_n(f)-f||_{L^2_T} \rightarrow 0, \,\, \text{ when } n \rightarrow \infty.
\end{equation}
The above convergence clearly holds when $f$ is continuous, and the general case follows by using the uniform continuity of $f$ and the density of $\cal{C}([0,T])$ in $L^2_T$.

\paragraph{Acknowledgement} We are grateful to Guy Barles, Rainer Buckdahn, Halim Doss, and  Nicole El Karoui for valuable discussions. We thank the anonymous referees for their helpful remarks. 


\small
\begin{thebibliography}{15}

\bibitem{B}  Barles, G.  Discontinuous viscosity solutions of first-order Hamilton-Jacobi equations: a guided visit, {\em Nonlinear Analysis, Theory, Methods and Applications},  20(9),  1123-1134, 1993.



 
 \bibitem{BJ} Barron E.N. and R. Jensen 
 Optimal Control and semicontinuous viscosity solutions, {\em Pfoceedings of the American Mathematical Society} {\bf 113}(2), (1991).

%


\bibitem{BP}  Barles, G., Perthame B. (1986) Discontinuous viscosity solutions of deterministic optimal control problems, {\em Decision and Control, 1986 25th IEEE Conference on Decision and Control}. 








\bibitem{BY} Bayraktar E. and  Yao S.,   
A Weak Dynamic Programming Principle for Zero-Sum Stochastic Differential Games with Unbounded Controls,  {\em SIAM J Control Optim}, 51(3), 2036--2080, 2013. 

\bibitem{BY11} Bayraktar E. and Yao, S. Optimal stopping for Non-linear Expectations,  {\em Stoch Proc Appl} (2011), 121(2), 185-211 and 212-264.



\bibitem{BS} Bertsekas D. and  Shreve S., 
{\em Stochastic Optimal Control: The Discrete Time Case}, Academic Press, Orlando (1978).
\bibitem{BN} Bouchard, B. and M. Nutz. Weak Dynamic Programming for Generalized State Constraints,
 {\em SIAM Journal on Control and Optimization}, 2012, 50(6), 3344-3373.

\bibitem{BT} Bouchard, B. and N. Touzi. Weak Dynamic Programming Principle for Viscosity Solutions,  {\em SIAM J Control Optim}, 2011, 49 (3), 948-962.


%


\bibitem{Buckdahn} Buckdahn, R. and Li, J., (2008), Stochastic Differential Games and Viscosity Solutions of Hamilton-Jacobi-Bellman-Isaacs Equations, 
	{\em SIAM J Control  Optim} 47(1), 444-475.
	
\bibitem{Buckdahn2} Buckdahn, R. and Li, J. (2009), Probabilistic interpretation for systems of Isaacs equations with two reflecting barriers, Nonlinear Differ.Equ. Appl. 16, pg. 381-420.

\bibitem{Buckdahn1} Buckdahn, R. and Nie, T. (2014), Generalized Hamilton-Jacobi-Bellman equations with Dirichlet boundary and stochastic exit time optimal control problem, 
http://arxiv.org/pdf/1412.0730v4.pdf





\bibitem{C}  Cohn, D. (2013).  {\em Measure Theory}, second edition, Birkhauser. 

%
\bibitem{Crauel} Crauel, H., {\em  Random Probability Measures on Polish Spaces}, Stochastics monographs, vol.11, Taylor and Francis,
London and New-York (2002).











\bibitem{DM1}
Dellacherie, C. and  Meyer, P.-A. (1975).
 \textit{Probabilit\'es et Potentiel, Chap. I-IV}. Nouvelle \'edition. Hermann. {\bf MR}{0488194}
 
\bibitem{DQS}  Dumitrescu, R., Quenez M.C., Sulem A., Optimal stopping for dynamic risk measures with jumps and obstacle problems, {\em Journal of Optimization Theory and Applications}, 
167(1) (2015),  219--242.

\bibitem{DQS2}  Dumitrescu, R., Quenez M.C., Sulem A. (2015), Mixed Generalized Dynkin Games and Stochastic control in a Markovian framework, http://arxiv.org/abs/1508.02742





%

\bibitem{10} El Karoui, N. , Kapoudjian, C., Pardoux, E., Peng, S. and Quenez, M-C.   Reflected solutions of backward SDE`s, and related obstacle problems for PDE's,  \textit{The Annals of Probability}, {\bf 25}(2), 702-737, 1997.

%

\bibitem{J}
Jacod, J. (1979).
 \textit{Calcul Stochastique et Probl\`emes de martingales}, Springer.





%
%

%
%


\bibitem{LP}  Li J. and S. Peng, Stochastic optimization theory of backward stochastic differential
equations with jumps and viscosity solutions of
Hamilton Jacobi Bellman equations,
{\em Nonlinear Analysis} 70 (2009) 1776-1796.

\bibitem{LS} Lions, P.L and P.E. Souganidis. 
Differential Games, Optimal Control and Directional Derivatives of Viscosity Solutions of Bellman's and Isaacs'  Equations, {\em SIAM J. Control and Optimization}, 23(4), 1985, 566--583. 






\bibitem{Peng92}  Peng, S. (1992),
A generalized dynamic programming principle and Hamilton-Jacobi-Bellman-Equation, 
 {\em Stochastics and Stochastics Reports}, 38.


\bibitem{15}  Peng, S. (2004),
Nonlinear expectations, nonlinear evaluations and risk measures,
165-253, {\em Lecture Notes in Math.}, 1856, Springer, Berlin.
%






\bibitem{16} Quenez M.-C. and Sulem A., BSDEs with jumps, optimization and applications to dynamic risk measures, {\em Stochastic Processes and Applications} 123 (2013) 3328-3357.

\bibitem{17} Quenez M.-C.  and  Sulem A.,Reflected BSDEs and robust optimal stopping for dynamic risk measures with jumps, {\em Stochastic Processes and Applications} 124, (2014) 3031--3054.

%

\bibitem{STZ} Soner M., Touzi N. and Zhang J., Wellposedness of second order backward SDEs, Probability Theory and Related Fields,  153, 149--190.
























\end{thebibliography}
\end{document}